\documentclass[a4paper, 11pt]{article}

\usepackage{amsmath, amscd, amsfonts, amssymb, amsthm, latexsym, url}
\usepackage{graphicx}
\usepackage{hyperref}
\usepackage{setspace}
\input{xy}
\xyoption{all}
\usepackage{enumerate}

\makeatletter
\def\url@leostyle{%
  \@ifundefined{selectfont}{\def\UrlFont{\sf}}{\def\UrlFont{\small\ttfamily}}}
\makeatother
\usepackage[left=3.5cm, right=3.5cm, top=3.5cm, bottom=3.5cm]{geometry}
\usepackage{sectsty}
\sectionfont{\large\bf}
\chapterfont{\large\bf\centering}
\chaptertitlefont{\LARGE\centering}
\partfont{\centering\bf}
\paragraphfont{\bf}

\theoremstyle{plain}
\newtheorem{thm}{Theorem}[section]
\newtheorem{lem}[thm]{Lemma}
\newtheorem{corollary}[thm]{Corollary}
\newtheorem{prop}[thm]{Proposition}

\theoremstyle{definition}
\newtheorem{df}[thm]{Definition}
\newtheorem{rmk}[thm]{Remark}
\newtheorem*{nota}{Notation}
\newtheorem{exa}[thm]{Example}

\newcommand{\roundb}[1]{\! \left(\! \left( #1 \right)\!\right)}
\newcommand{\squareb}[1]{\! \left[\! \left[ #1 \right]\!\right]}
\newcommand{\curlyb}[1]{\! \left\{\! \left\{ #1 \right\}\!\right\}}

\newcommand{\OO}{\mathcal{O}}
\newcommand{\OOt}{\widetilde{\mathcal{O}}}
\newcommand{\pp}{\mathfrak{p}}

\newcommand{\LL}{\mathcal{L}}

\newcommand{\QQ}{\mathbb{Q}}
\newcommand{\ZZ}{\mathbb{Z}}
\newcommand{\KK}{\mathbb{K}}
\newcommand{\Kt}{\widetilde{K}}

\newcommand{\Vect}{\mathbf{Vect}}

\newcommand{\car}{\mathrm{char}\:}
\newcommand{\Frac}{\mathrm{Frac}\:}
\newcommand{\Tr}{\mathrm{Tr}\:}

\newcommand{\into}{\hookrightarrow}

\newcommand{\arxivmath}[1]{\href{http://arxiv.org/abs/math/#1}{{\tt \small arXiv:math/#1}}}
\newcommand{\arxiv}[1]{\href{http://arxiv.org/abs/#1}{{\tt \small arXiv:#1}}}

\usepackage{todonotes}
\title{Functional analysis on two-dimensional local fields} % This is the full title of the paper
\author{Alberto C\'amara\footnote{The author is supported by a Doctoral Training Grant at the University of Nottingham.}}
\date{\today}
\begin{document}

\maketitle

\begin{abstract}
  We establish how a two-dimensional local field can be described as a
  locally convex space once an embedding of a local field into it has
  been fixed. We study the resulting spaces from a functional analytic
  point of view: in particular we study bounded, c-compact and
  compactoid submodules of two-dimensional local fields.
\end{abstract}

%% Should comment this paragraph out when compiling for presentation. 
%\listoftodos
%\tableofcontents

\section*{Introduction}

This work is concerned with the study of characteristic zero
two-dimensional local fields. These are complete discrete valuation
fields whose residue field is a local field, either of characteristic
zero or positive.

Following an idea introduced in
\cite{morrow-explicit-approach-to-residues}, we
do not regard two-dimensional fields as fields in the usual sense, but as
an embedding of fields $K \into F$, where $K$ is a local field and $F$ is
a two-dimensional local field. Given a two-dimensional local field $F$,
such field embeddings always exist and we are not assuming any extra
conditions on $F$; we are only changing our point of view.

In the arithmetic-geometric context, such field embeddings arise in the
following way: suppose that $S$ is the spectrum of the ring of integers
of a number field and that $f: X \to S$ is an arithmetic surface (for our
purposes it is enough to suppose that $X$ is a regular 2-dimensional
scheme and that $f$ is projective and flat). Choose a closed point $x \in X$
and an irreducible curve $\overline{\left\{ y \right\}} \subset X$ such
that $x$ is regular in $\overline{ \left\{ y \right\} }$, and let $s = f(x) \in S$. Starting from the
local ring of regular functions $\OO_{X,x}$, we obtain a two-dimensional
local field $F_{x,y}$ through a process of repeated completions and
localizations:
\begin{equation*}
  F_{x,y} = \Frac\left( \widehat{ \left( \widehat{\OO_{X,x}}\right)_y } \right).
\end{equation*}
This is analogue to the procedure of completion and
localization that allows us to obtain a local field $K_s =
\Frac(\widehat{\OO_{S,s}})$ from the closed point $s \in S$. The structure morphism $\OO_{S,s} \to \OO_{X,x}$ 
induces a field embedding $K_s \into F_{x,y}$.

The moral of the above paragraph is that if two-dimensional fields arise
from an arithmetic-geometric context then they always come with a prefixed
local field embedded into them.

\vskip .5cm

What we study in this work is the $K$-vector space structure associated
to $F$ via the embedding $K\into F$. As such, we connect the topological
theory of two-dimensional local fields with the theory of nonarchimedean
locally convex vector spaces. In particular, for the fields
$K\roundb{t}$ and $K\curlyb{t}$ (see \textsection
\ref{sec:categoryof2dlfs} for the definition of the latter), we establish in Corollaries
\ref{cor:seminormsequi} and \ref{cor:seminormsmixedstdcase}, a family of defining seminorms for the
higher topology of the form
\begin{equation*}
  \left\| \sum_{i} x_i t^i \right\| = \sup_i |x_i| q^{n_i},
\end{equation*}
where $( n_i )_{i \in \ZZ} \subset \ZZ \cup \left\{ -\infty
\right\}$ is a sequence subject to certain conditions and $q$ is the
number of elements in the residue field of $K$.

In particular, this provides us with a new way to
describe higher topologies on two-dimensional local fields which does
not rely on taking a lifting map from the residue field as in
\cite{madunts-zhukov-topology-hlfs}. This also introduces a new concept
of bounded subset.

Although our description of higher topologies is valid for both equal and 
mixed characteristic
two-dimensional local fields, the study of the functional theoretic
properties in the two cases suggests that similarities stop here.
Equal characteristic fields may be shown to be LF-spaces (see
\textsection \ref{sec:lcspoverK} for the definition and Corollary
\ref{cor:propsofequicar}) and,
as such, they are bornological, nuclear and reflexive.
This characterization is
unavailable for mixed characteristic fields such as $K\curlyb{t}$ and we show how these
properties do not hold.

One of the advantages of our point of view is that certain submodules 
of $F$ arise as the families of c-compact and compactoid submodules, and
therefore have a property which is a linear-topological analogue of
compactness. In particular, compactoid submodules coincide with bounded
submodules in equal characteristic (this is a consequence of nuclearity)
 and define a family strictly contained
in that of bounded submodules in mixed characteristic. By using the associated
bornology we achieve in Theorem \ref{thm:selfduality} a very explicit
self-duality result in the line of \cite[Remark in \textsection 3]{fesenko-aoas1}.

\vskip .5cm

We briefly outline the contents and main results of this work. Sections \textsection \ref{sec:lcspoverK} and \textsection
\ref{sec:categoryof2dlfs} summarise relevant parts of the theory of
nonarchimedean locally convex vector spaces and the structure of
two-dimensional local fields, respectively. We have included them in this
work in
order to be able to refer to certain general results in later parts of
the work and in order to fix notations and conventions. Hence, we do not
supply proof for any statement in these sections, but refer the reader to
the available literature. 

At the beginning of section \textsection \ref{sec:highertopislocconv} we
center our attention in the $K$-vector spaces $K\roundb{t}$ and
$K\curlyb{t}$, two examples of two-dimensional local fields whose
topological behaviour determines the topological properties of equal
characteristic and mixed characteristic two-dimensional local fields,
respectively. We
prove in Propositions \ref{prop:locconvequicase} and
\ref{prop:locconvmixedcase} how the higher topology on these two fields induces
the structure of a locally convex $K$-vector space and we describe these
locally convex topologies in terms of seminorms in Corollaries
\ref{cor:seminormsequi} and \ref{cor:seminormsmixedstdcase}. We can
easily study the case of $K\roundb{t}$, as in Proposition
\ref{prop:toponKsquarebisprodtop} we prove that the higher topology
defines an LF-space and deduce the main analytic properties of
this space in Corollary \ref{cor:propsofequicar}: it is complete,
bornological, barrelled, reflexive and nuclear. In subsequent sections,
we show how $K\curlyb{t}$ is not bornological (Corollary
\ref{cor:curlybisnotbornological}), barrelled (Proposition
\ref{prop:curlybisnotbarrelled}), nuclear (Corollary
\ref{cor:curlybisnotnuclear}) nor reflexive (Corollary
\ref{cor:curlybisnotreflexive}).

Section \textsection \ref{sec:bornology} deals with the nature of bounded
subsets of the considered locally convex spaces. In particular, Propositions
\ref{prop:basicboundedsetequicar} and \ref{prop:basicboundedsetmixedcar}
describe a basis for the convex bornology associated to the locally
convex topology, usually known as the Von-Neumann bornology. We also prove in Proposition \ref{prop:muisbounded} that
the multiplication map on the two-dimensional local field, despite not
being continuous, is bounded.

\textsection
\ref{sec:completeccompactcompactoidsubmodules} contains a study of
relevant $\OO$-submodules of $K\roundb{t}$ and $K\curlyb{t}$ such as
rings of integers and rank-two rings of integers. In the case of
$K\roundb{t}$, these may be shown to be c-compact (Proposition
\ref{prop:Ksquarebisccpct} and Corollary \ref{cor:ranktwoequicarisccpct})
but not compactoid (Corollary
\ref{cor:ringsofintegersequicarnotcompactoid}). In the case of
$K\curlyb{t}$, these rings of integers may be shown not to be compactoid
nor c-compact (Corollary \ref{cor:ringsofintegersmixedcarnotcompactoid}). It follows from
nuclearity that all bounded submodules of $K\roundb{t}$ are compactoid.
However, this is not the case for $K\curlyb{t}$: in Proposition
\ref{prop:basiccompactoidsubmodulesmixedchar} we describe a basis for the
bornology of compactoid submodules on $K\curlyb{t}$, which is strictly
coarser than the Von-Neumann bornology. We are however able to prove that
boundedness of the multiplication map on $K\curlyb{t}$ holds for this
coarser bornology (Corollary \ref{cor:muisboundedcompactoid}).

In \textsection \ref{sec:duality} we study duality issues. In particular,
Theorem \ref{thm:selfduality} establishes that the two-dimensional local
fields considered are isomorphic in the category of locally convex vector
spaces to their appropriately topologized duals: we deduce some
consequences of this fact. Finally, we study polarity issues after
identifying our two-dimensional local fields and their duals.

In \textsection \ref{sec:generalcase}, we extend the results of the
previous sections to the case of a general embedding $K \into F$ of a
local field into a two-dimensional local field. It is important to remark
that the functional analytic properties of equal characteristic
two-dimensional local fields are the same as $K\roundb{t}$ and the
properties of mixed characteristic two-dimensional local fields closely
resemble those of $K\curlyb{t}$.

Sections \textsection \ref{sec:archcase} and \textsection
\ref{sec:charpcase} explain how the results in this work can also be
applied to archimedean two-dimensional local fields and positive
characteristic local fields, respectively. In the first case, we are
dealing with LF-spaces and we deduce our results from the
well-established theory of (archimedean) locally convex spaces. In the
second case, we relate the locally convex structure of vector spaces over
$\mathbb{F}_q\roundb{u}$ to the linear topological structure of vector
spaces over $\mathbb{F}_q$ through restriction of scalars. The study of
two-dimensional local fields in positive characteristic using linear
topological tools had been started by Parshin \cite{parshin-lCFT}, and
our point of view links with his in this case.

Finally, we discuss some applications and further directions of research
in \textsection \ref{sec:future}.

\begin{nota}
Whenever $F$ is a complete discrete valuation
field, we will denote by $\OO_F, \pp_F, \pi_F$ and $\overline{F}$ its
ring of integers, the unique nonzero prime ideal in the ring of integers,
an element of valuation one and the residue field, respectively. Whenever
$x \in \OO_F$, $\overline{x} \in \overline{F}$ will denote its image
modulo $\pp_F$. A {\it two-dimensional local field} is a complete discrete valuation field $F$ such that $\overline{F}$ is a local field.

Throughout the text, $K$ will denote a characteristic zero local
field, that is, a finite extension of $\QQ_p$ for some prime $p$. The cardinality of the finite field $\overline{K}$ will be denoted by $q$. The absolute value of $K$
will be denoted by $|\cdot|$, normalised so that $|\pi_K| = q^{-1}$. Due
to far too many appearances in the text, we will ease notation by
letting $\OO := \OO_K$, $\pp := \pp_K$ and $\pi := \pi_K$.

The conventions $\pp^{-\infty} = K$, 
$\pp^\infty = \left\{ 0 \right\}$ and $q^{-\infty} = 0$ will be used.
\end{nota}

\paragraph{Acknowledgements.} I am indebted to Matthew Morrow and Oliver
Br\"aunling, with whom
I had the initial discussions that later turned into this piece of work.
Thomas Oliver has been my counterpart in many interesting conversations
during the process of establishing and writing down the results contained
here.
I am also in great debt with Ralf Meyer and Cristina P\'erez-Garc\'ia,
whose comments on early drafts of this work have proven to be invaluable.
I would also like to thank the anonymous referee, whose comments and
criticism have led me to improve the text considerably.
Finally, I thank my supervisor Ivan Fesenko for his guidance and
encouragement.

\section{Locally convex spaces over $K$}
\label{sec:lcspoverK}

In this section we summarise some concepts and fix some notation
regarding locally convex vector spaces over $K$. This is both for the
reader's convenience as much as for establishing certain statements and
properties for later reference. 

The theory of locally convex vector spaces over a nonarchimedean field is
well developed in the literature, so we will keep a concise exposition of
the facts that we will require later. Both
\cite{schneider-non-archimedean-functional-analysis} and
\cite{perez-garcia-schikof-locally-convex-spaces-nonarchimedean-valued-fields}
are very good references on the topic.

Let $V$ be a $K$-vector space. A {\it lattice} in $V$ is an $\OO$-submodule
$\Lambda \subseteq V$ such that for any $v \in V$ there is an element $a
\in K^\times$ such that $av \in \Lambda$. This is equivalent to having
\begin{equation*}
  \Lambda \otimes_{\OO} K \cong V
\end{equation*}
as $K$-vector spaces. A subset of $V$ is said to be {\it convex} if it is of
the form $v + \Lambda$ for $v \in V$ and $\Lambda$ a lattice in $V$. A
vector space
topology on $V$ is said to be {\it locally convex} if the filter of
neighbourhoods of zero admits a collection of lattices as a basis. 

A {\it seminorm} on $V$ is a map $\| \cdot \|: V \to \mathbb{R}$ such
that:
\begin{enumerate}
  \item $\|\lambda v \| = |\lambda| \cdot \|v\|$ for every $\lambda \in K$, $v \in V$,
  \item $\| v + w \| \leq \max \left( \|v\|, \|w\| \right)$ for all
    $v, w \in V$.
\end{enumerate}
These conditions imply in particular that a seminorm only takes
non-negative values and that $\|0\| = 0$. A seminorm $\| \cdot \|$ is
said to be a {\it norm} if $\|x\|=0$ implies $x=0$.

%In other words, in order to have a locally convex topology on $V$, it is enough to describe a family of lattices $\left\{ \Lambda_j \right\}_{j \in J}$ satisfying the two conditions:
%\begin{enumerate}
%  \item For every index $j \in J$ and $a \in K^\times$, there is an index $k \in J$ such that $\Lambda_k \subseteq a \Lambda_j$.
%  \item Given a couple of indices $i,j \in J$, there is an index $k \in J$ such that $\Lambda_k \subseteq \Lambda_i \cap \Lambda_j$.
%\end{enumerate}

The {\it gauge seminorm} of a lattice $\Lambda \subseteq V$ is defined by the
rule:
\begin{equation}
  \label{eqn:gaugeseminorm}
  \| \cdot \|_{\Lambda} : V \to \mathbb{R},\quad v \mapsto \inf_{v \in
  a\Lambda} \vert a \vert.
\end{equation}

Given a family of seminorms $\left\{ \| \cdot \|_j \right\}_{j \in J}$ on
$V$, there is a unique coarsest vector space topology on $V$ making 
the maps $\| \cdot \|_j : V \to \mathbb{R}$ continuous for every $j \in
J$. Such topology is locally convex: since the intersection of a finite
number of lattices is a lattice, the {\it closed balls} 
\begin{equation*}
  B_j(\varepsilon) = \left\{ v \in V;\; \|v\|_j \leq \varepsilon
  \right\}, \quad \varepsilon \in \mathbb{R}_{>0}, j \in J
\end{equation*}
supply a subbasis of neighbourhoods of zero
consisting of open lattices. Note that the use of the adjective {\it
closed} here is, as usual in this setting, an imitation of the analogous
archimedean convention. Topologically, $B_j(\varepsilon)$ and $\left\{
v \in V;\; \|v \|_j < \varepsilon \right\}$ are both open and closed.

A locally convex topology can be described in terms of lattices or in terms of seminorms; passing from one point of view to the other is a simple matter of language.

A locally convex vector space $V$ is said to be {\it normable} if its
topology may be defined by a single norm. By saying that $V$ is {\it normed}
$K$-vector space, we will imply that we are considering a norm on it and
that we regard the space together with the locally convex topology defined by a norm.

  For a locally convex vector space $V$, a subset $B \subset V$ is {\it bounded} if for any open lattice $\Lambda
  \subset V$ there is an $a \in K$ such that $B \subseteq a\Lambda$. Alternatively, $B$ is bounded if for every continuous seminorm $\| \cdot \|$ on $V$ we have
  \begin{equation*}
    \sup_{v \in B} \|v\| < \infty.
  \end{equation*}
  A locally convex $K$-vector space $V$ is {\it bornological} if any
  seminorm which is bounded on bounded sets is continuous. A linear map
  between locally convex vector spaces $V \to W$ is said to be {\it
  bounded} if the image of any bounded subset of $V$ is a bounded subset
  of $W$.

  More generally, a {\it bornology} on a set $X$ is a collection
  $\mathcal{B}$ of
  subsets of $X$ which cover $X$, is hereditary under inclusion and stable
  under finite union. We say that the elements of $\mathcal{B}$ are
  {\it bounded sets} and the pair $(X,\mathcal{B})$ is referred to as a
  {\it bornological space} \cite[Chapter I]{hogbe-nlend-bornologies-functional-analysis}.

Just like a topology on a set is the minimum
amount of information required in order to have a notion of open set 
and continuous map, a bornology on a set is the minimum amount of 
information required in order
to have a notion of bounded set and {\it bounded map}, the latter being a
map between two bornological spaces which preserves bounded sets. A {\it
basis} for a bornology $\mathcal{B}$ on a set is a subfamily
$\mathcal{B}_0 \subset \mathcal{B}$ such that every element of
$\mathcal{B}$ is contained in an element of $\mathcal{B}_0$.

The bornology which we have described above for a locally convex vector
space $V$ is known as the {\it Von-Neumann bornology}
\cite[\textsection I.2]{hogbe-nlend-theorie-des-bornologies-applications}, and it is
compatible with the vector space structure, meaning that the vector space
operations are bounded maps. Moreover, the Von-Neumann bornology on
a locally convex vector space is {\it convex}, as it admits a basis given
by convex subsets \cite[\textsection I.6]{hogbe-nlend-theorie-des-bornologies-applications}.

%On an arbitrary vector space, if a convex bornology is specified, we get
%a bornological locally convex space by considering the strongest locally
%convex topology that gives rise to the specified family of bounded
%sets.

Open lattices in a non-archimedean locally convex space are also closed
\cite[textsection 6]{schneider-non-archimedean-functional-analysis}.
A locally convex space $V$ is said to be {\it barrelled} if any closed lattice is open.

\vskip .5cm

Among many general ways to construct locally convex spaces \cite[\textsection 5]{schneider-non-archimedean-functional-analysis}, we will require the use of products.

\begin{prop}
\label{prop:producttopoflocconv}
Let $\left\{ V_i \right\}_{i \in I}$ be a family of locally convex $K$-vector spaces, and let $V = \prod_{i \in I} V_i$. Then the product topology on $V$ is locally convex.
\end{prop}

\begin{proof}
  See \cite[\textsection 5.C]{schneider-non-archimedean-functional-analysis}.
If $\left\{ \Lambda_{i,j} \right\}_{j}$ denotes the set of open lattices
of $V_i$ for $i \in I$, then the set of open lattices of $V$ is given by
finite intersections of lattices of the form $\pi_i^{-1}\Lambda_{i,j}$. 

Equivalently, the product topology on $V$ is the one defined by all seminorms of the form 
\begin{equation*}
  v \mapsto \sup_{i,j} \| \pi_i(v) \|_{i,j},
\end{equation*}
where $\left\{ \| \cdot \|_{i,j} \right\}_j$ is a defining
family of seminorms for $V_i$ for all $i \in I$, $\pi_i: V \to V_i$ is
the corresponding projection and the supremum is taken over a finite
collection of indices $i,j$.
\end{proof}

Similarly, if $(X_i, \mathcal{B}_i)_{i \in I}$ is a collection of
bornological sets, the {\it product bornology} on $X =\prod_{i \in I}
X_i$ is the one defined by taking as a basis the sets of the form
$B = \prod_{i \in I} B_i$ with $B_i \in \mathcal{B}_i$
\cite[\textsection 2.2]{hogbe-nlend-bornologies-functional-analysis}.

\vskip .5cm

Another construction which we will require is that of inductive limits.
Let $V$ be a $K$-vector space and $\left\{ V_i \right\}_{i \in I}$ be a
collection of locally convex $K$-vector spaces. Let, for each $i \in I$,
$f_i: V_i \to V$ be a $K$-linear map. The final topology for the
collection $\left\{ f_i \right\}_{i \in I}$ is not locally convex in
general. However, there is a finest locally convex topology on $V$ making
the map $f_i$ continuous for every $i \in I$
\cite[\textsection 5.D]{schneider-non-archimedean-functional-analysis}. That topology is called the
{\it locally convex final topology} on $V$. Inductive limits and direct sums 
of locally convex spaces are particular examples of such construction.

\begin{df}
  \label{df:strictdirectlim}
  Suppose that $V$ is a $K$-vector space and that we have an increasing
  sequence of vector subspaces $V_1 \subseteq V_2 \subseteq \cdots
  \subseteq V$ such that $V = \cup_{n \in \mathbb{N}} V_n$. Suppose that
  for each $n \in \mathbb{N}$, $V_n$ is equipped with a locally convex topology such
  that $V_n \into V_{n+1}$ is continuous. Then the final locally convex topology on $V$
  is called the {\it strict inductive limit topology}.
\end{df}

Let us fix, from now until the end of the present section, a locally
convex vector space $V$. 
In order to discuss completeness issues, we require to deal not only with
sequences, but arbitrary nets. 

Let $I$ be a directed set. A
{\it net} in $V$ is a family of vectors $\left( v_i \right)_{i \in I} \subset
V$. A {\it sequence} is a net which is indexed by the set of
natural numbers.
The net $\left( v_i \right)_{i \in I}$ {\it converges} to a vector $v$, and we shall write
$v_i \to v$, if for any $\varepsilon > 0$ and continuous seminorm $\| \cdot \|$ on
$V$, there is an index $i \in I$ such that for every $j \geq i$ we have
$\| v_j - v \| \leq \varepsilon$.
Similarly, the net $\left( v_i \right)_{i \in I}$ is said to be {\it
Cauchy} if
for any $\varepsilon >0$ and continuous seminorm $\| \cdot \|$ on $V$ there is an
index $i \in I$ such that for every pair of indices $j,k \geq i$ we have
$\| v_j-v_k \| \leq \varepsilon$.

\begin{df}
  A subset $A \subseteq V$ is said to be {\it complete} if any Cauchy net in
  $A$ converges to a vector in $A$.
\end{df}

A {\it $K$-Banach space} is a complete normed locally convex vector
space. $V$ is said to be a Fr\'echet space if it is complete and its
locally convex topology is metrizable. A locally convex vector space is
said to be an {\it LF-space} if it may be constructed as an inductive limit
of a family of Fr\'echet spaces.

\begin{exa}
  $K$ is a Fr\'echet space. There is a unique locally convex topology on
  any finite dimensional $K$-vector space which defines a structure of
  Fr\'echet space \cite[Proposition 4.13]{schneider-non-archimedean-functional-analysis}.
\end{exa}

In general, the usual topological notion of compactness is not very 
powerful for the study of infinite dimensional vector spaces over 
non-archimedean fields. This is why we prefer to use the language of
c-compactness, which is an $\OO$-linear concept of compactness.

\begin{df}
  Let $A$ be an $\OO$-submodule of $V$. $A$ is said to be {\it
  $c$-compact} if, for any decreasing filtered family $\left\{ \Lambda_i \right\}_{i \in I}$ of open lattices of $V$, the canonical map
  \begin{equation*}
    A \to \varprojlim_{i \in I} A/\left( \Lambda_i \cap A \right)
  \end{equation*}
  is surjective.
\end{df}

\begin{exa}
  \label{exa:Kisccpt}
  The base field $K$ is c-compact as a $K$-vector space
  \cite[\textsection 12]{schneider-non-archimedean-functional-analysis}. This shows that a c-compact module need not be bounded.
\end{exa}

This property may be phrased in a more topological way.

\begin{prop}
  \label{prop:ccompactness}
  An $\OO$-submodule $A \subseteq V$ is $c$-compact if and only if for any
  family $\left\{ C_i \right\}_{i \in I}$ of closed convex subsets $C_i
  \subseteq A$ such that $\bigcap_{i \in I }C_i = \emptyset$ there are
  finitely many indices $i_1, \ldots, i_m \in I$ such that $C_{i_1} \cap \ldots \cap C_{i_m} = \emptyset$.
\end{prop}

\begin{proof}
  See \cite[Lemma 12.1.ii and subsequent paragraph]{schneider-non-archimedean-functional-analysis}.
\end{proof}

\begin{prop}
  \label{prop:prodccptisccpt}
  Let $\left\{ V_h \right\}_{h \in H}$ be a collection of locally convex
  $K$-vector spaces, and for each $h \in H$ let $A_h \subseteq V_h$ be a c-compact $\OO$-submodule. Then $\prod_{h \in H} A_h$ is c-compact in $\prod_{h\in H} V_h$.
\end{prop}

\begin{proof}
  \cite[Prop. 12.2]{schneider-non-archimedean-functional-analysis}.
\end{proof}

Another notion which is used in this setting is that of a compactoid
$\OO$-module; it is a notion which is analogous to that of relative
compactness in the archimedean setting.

\begin{df}
  Let $A \subseteq V$ be an $\OO$-submodule. $A$ is {\it compactoid} if for any
  open lattice $\Lambda$ of $V$ there are finitely many vectors
  $v_1, \ldots, v_m \in V$ such that
  \begin{equation*}
    A \subseteq \Lambda + \OO v_1 + \cdots + \OO v_m.
  \end{equation*}
\end{df}

Let $A \subseteq V$ be an $\OO$-submodule. If $A$ is c-compact, then it
is closed and complete. Similarly, if $A$ is compactoid then it is
bounded. \cite[\textsection 12]{schneider-non-archimedean-functional-analysis}.

\begin{prop}
  \label{prop:ccpctbddcptoidcmplt}
  Let $A \subseteq V$ be an $\OO$-submodule. The following are equivalent.
  \begin{enumerate}
    \item $A$ is c-compact and bounded.
    \item $A$ is compactoid and complete.
  \end{enumerate}
\end{prop}

\begin{proof}
  \cite[Prop. 12.7]{schneider-non-archimedean-functional-analysis}.
\end{proof}

The collection of compactoid $\OO$-submodules of $V$ generates a
bornology which is a priori weaker than the one given by the locally convex
topology.

\begin{rmk}
It should be pointed out that the locally convex vector spaces
that we consider in this work are always defined over a local field,
which is discretely valued and, therefore, locally compact and
spherically complete. This implies that the general theory of locally
convex spaces over a nonarchimedean complete field simplifies in our
setting. In particular, for an $\OO$-submodule $A \subseteq V$,
compactoidness and completeness imply compactness
\cite[Theorem 3.8.3]{perez-garcia-schikof-locally-convex-spaces-nonarchimedean-valued-fields}. We choose, however, to
use the language of c-compact and compactoid submodules.
\end{rmk}

If $V, W$ are two locally convex $K$-vector spaces, a linear map $f: V
\to W$ is continuous as soon as the pull-back of a continuous seminorm is
a continuous seminorm. We denote the
$K$-vector space of continuous linear maps between $V$ and $W$ by
$\LL(V,W)$.

The space $\LL(V,W)$ may be topologized in the following way. Let
$\mathcal{B}$ be a collection of bounded subsets of $V$. For any continuous
seminorm $\| \cdot \|$ on $W$ and $B \in \mathcal{B}$, consider the seminorm
\begin{equation*}
  \| \cdot \|_B : \LL(V,W) \to \mathbb{R},\quad f \mapsto \sup_{v \in B}
  \| f(v) \|.
\end{equation*}

\begin{df}
  \label{df:toponspacesoflinearmaps}
  We write $\LL_\mathcal{B}(V,W)$ for the space of
  continuous linear maps from $V$ to $W$ endowed with the locally convex
  topology defined by the seminorms $\| \cdot \|_B$, for every continuous
  seminorm $\| \cdot \|$ on $W$ and $B \in \mathcal{B}$.

  In the particular case in which $\mathcal{B}$ consists of all bounded
  sets of $V$, we write $\LL_b(V,W)$ for the resulting space, which
  is then said to have the topology of {\it uniform convergence}, or
  {\it b-topology}. If
  $\mathcal{B}$ consists only of the singletons $\left\{ v
  \right\}$ for $v \in V$, we denote the resulting space by
  $\LL_s(V,W)$ and say that it has the topology of {\it point-wise
  convergence}.
  Finally, if $\mathcal{B}$ is the collection of compactoid
  $\OO$-submodules of $V$, we denote the resulting space by
  $\LL_c(V,W)$ and say that it has the topology of {\it uniform
  convergence on compactoid submodules}, or {\it c-topology}.
\end{df}

There are two cases of particular interest: the topological dual space $V' =
\LL(V,K)$, and the endomorphism ring $\LL(V) = \LL(V,V)$. We denote
$F'_s, F'_b, F'_c$, $\LL_s(V), \LL_b(V)$ and $\LL_c(V)$ for
the corresponding topologies of point-wise convergence, uniform
convergence and uniform convergence on compactoid submodules, respectively.

The choice of a family of bounded subsets $\mathcal{B}$ of $F$ does not
affect $F'_{\mathcal{B}}$ as a set, but it does affect the bidual space. As such, in the category of locally convex vector spaces over $K$, it is an
interesting issue to classify which spaces are isomorphic, algebraically
and/or topologically, to certain bidual spaces through the duality maps
\begin{equation}
  \delta: V \to (V'_\mathcal{B})', \quad v \mapsto \delta_v(l) = l(v).
  \label{eqn:dualitymap}
\end{equation}
The best possible case is when $\delta$ induces a topological isomorphism
between $V$ and $(V'_b)'_b$; in this case we say that $V$ is {\it
reflexive}.

\begin{prop}
  \label{prop:reflexiveisbarrelled}
  Every locally convex reflexive $K$-vector space is barrelled.
\end{prop}

\begin{proof}
  \cite[Lemma 15.4]{schneider-non-archimedean-functional-analysis}.
\end{proof}

The notion of polarity plays a role in the study of duality, as it
provides us with a way of relating $\OO$-submodules of $V$ to
$\OO$-submodules of $V'$.

\begin{df}
  \label{df:pseudopolar}
  If $A \subseteq V$ is an $\OO$-submodule, we define its {\it
  pseudo-polar} by
  \begin{equation*}
    A^p = \left\{ l \in V';\; |l(v)| < 1 \text{ for all } v \in A \right\}.
  \end{equation*}
  The pseudo-bipolar of $A$ is
  \begin{equation*}
    A^{pp} = \left\{ v \in V;\; |l(v)| < 1 \text{ for all } l \in A^p
    \right\}.
  \end{equation*}
\end{df}

Taking the pseudo-polar of an $\OO$-submodule of $V$ gives an
$\OO$-submodule of $V'$.

We have that $l \in A^p$ if and only if $l(A) \subseteq \pp$. Note
that the traditional notion of polar relaxes the condition in the
definition of pseudo-polar to $|l(v)| \leq 1$ or, equivalently,
$l(A) \subseteq \OO$. Introducing the distinction
is an important technical detail, as pseudo-polarity is a better-behaved
notion in the nonarchimedean setting.

\begin{prop}
  \label{prop:pseudopolar}
  Let $A \subseteq V$ be an $\OO$-submodule. We have
  \begin{enumerate}
    \item If $A \subseteq B \subseteq V$ is another $\OO$-submodule, then
      $B^p \subseteq A^p$.
    \item $A^p$ is closed in $V'_s$.
    \item Let $\mathcal{B}$ be any collection of bounded subsets of $V$.
      If $A \in \mathcal{B}$, then $A^p$ is an open lattice in
      $V'_{\mathcal{B}}$.
    \item $A^{pp}$ is equal to the closure of $A$ in $V$.
  \end{enumerate}
\end{prop}

\begin{proof}
  Statements (i), (ii) and (iii) are part of \cite[Lemma
  13.1]{schneider-non-archimedean-functional-analysis}. (iv) is
  \cite[Proposition 13.4]{schneider-non-archimedean-functional-analysis}.
\end{proof}

\vskip .5cm

In order to conclude this section we define nuclear spaces. For any
submodule $A \subseteq V$, denote $V_A := A \otimes_{\OO} K$,
endowed with the locally convex topology associated to the gauge seminorm
$\| \cdot \|_A$. $V_A$ may not be a Hausdorff space, but its
completion
\begin{equation*}
  \widehat{V_A} := \varprojlim_{n \in \ZZ} V_A / \pi^n A
\end{equation*}
is a $K$-Banach space.

\begin{df}
  \label{df:nuclearspace}
  $V$ is said to be {\it nuclear} if for any open lattice $\Lambda \subseteq V$
  there exists another open lattice $M \subseteq \Lambda$ such that the
  canonical map $\widehat{V_M} \to \widehat{V_\Lambda}$ is compact, that
  is: there is an open lattice in $\widehat{V_M}$ such that the closure
  of its image is bounded and c-compact.
\end{df}

\begin{prop}
  We have:
  \label{prop:propsofnuclearspaces}
  \begin{enumerate}
    \item An $\OO$-submodule of a nuclear space is bounded if and only if
      it is compactoid.
    \item Arbitrary products of nuclear spaces are nuclear.
    \item Strict inductive limits of nuclear spaces are nuclear.
  \end{enumerate}
\end{prop}

\begin{proof}
  (i) is \cite[Proposition
  19.2]{schneider-non-archimedean-functional-analysis}, (ii) is
  \cite[Proposition 19.7]{schneider-non-archimedean-functional-analysis}
  and (iii) is \cite[Corollary
  19.8]{schneider-non-archimedean-functional-analysis}.
\end{proof}

\section{Our point of view on two-dimensional local fields}
\label{sec:categoryof2dlfs}

We consider the category whose objects are field inclusions
\begin{equation*}
  K \hookrightarrow F
\end{equation*}
where $K$ is our fixed characteristic zero local field and $F$ is a
two-dimensional local field. In such case, we shall say that $F$ is a
two-dimensional local field over $K$. A morphism in this category between $K \hookrightarrow F_1$ and $K \hookrightarrow F_2$ is a commutative diagram of field inclusions
\begin{equation*}
  \xymatrix{
  F_1 \ar@{->}[r] & F_2 \\
  K  \ar@{->}[u] \ar@{->}[ur] & 
  }
\end{equation*}
where $F_1 \hookrightarrow F_2$ is an extension of complete discrete valuation fields.

The classification of characteristic zero two-dimensional local fields follows from Cohen
structure theory of complete local rings and was established in
\cite{madunts-zhukov-topology-hlfs}. The particular case with which we
are dealing is very well described in
\cite[\textsection 2.2 and 2.3]{morrow-explicit-approach-to-residues}.

By this classification, given a two-dimensional local field $F$ it is
always possible to exhibit a local field contained in it, so our
assumption does not imply any further structure on $F$. Let us briefly recall the structure of two-dimensional local fields, which depends on the relation between the characteristics of $F$ and $\overline{F}$.

If $\car F = \car \overline{F}$, the choice of a field embedding
$\overline{F} \into F$ determines an isomorphism $F \cong
\overline{F}\roundb{t}$ \cite[\textsection
II.5]{fesenko-vostokov-local-fields}. Such an isomorphism is not unique,
as it does indeed depend on the choice of coefficient field $\overline{F} \into F$. 

Besides fields of Laurent series, there is another construction which is
key in order to work with two-dimensional local fields, and higher local
fields in general. For any complete discrete valuation field $L$, consider
\begin{equation*}
  %\label{eqn:standardmixed2dlf}
  L\curlyb{t} = \left\{ \sum_{i \in \ZZ} x_i t^i;\; x_i \in L,\;
  \inf_{i \in \ZZ} v_{L}(x_i) > -\infty,\; x_i \to 0 \;(i \to -\infty)\right\},
\end{equation*}
with operations given by the usual addition and multiplication of power series. Note that we need to use convergence of series in $L$ in order to define the product. With the discrete valuation given by
\begin{equation*}
  v_{L\curlyb{t}}\left( \sum_{i\in \ZZ} x_i t^i \right) := \inf v_L(x_i),
\end{equation*}
$L\curlyb{t}$ turns into a complete discrete valuation field. In the
particular case in which $L$ is a characteristic zero local field, the
field $L \curlyb{t}$ is a 2-dimensional local field which we call
the {\it standard mixed characteristic field over $L$}. Its first residue field
is $\overline{L}\roundb{\overline{t}}$.

We view elements of $L$ as elements of $L\curlyb{t}$ in the obvious way.
In particular, if $\pi_L$ is a uniformizer of $\OO_L$, it is also a
uniformizer of $\OO_{L\curlyb{t}}$; the element $t \in L\curlyb{t}$ is
such that $\overline{t} \in \overline{L}\roundb{\overline{t}}$ is a
uniformizer.

Suppose now that $F$ is any two-dimensional local field such that $\car F
\neq \car \overline{F}$. Then there is a unique field embedding $\QQ_p
\into F$. Let $\Kt$ be the algebraic closure of $\QQ_p$ in $F$; it is a
finite extension on $\QQ_p$. In this case, $F$ contains a subfield which
is $\Kt$-isomorphic to $\Kt\curlyb{t}$, the extension
$\Kt\curlyb{t} \into F$ being finite. Finally, if a field embedding $K
\into F$ with $K$ a local field is given, then we have $\QQ_p \subseteq K
\subseteq \Kt$ and the extension $K\curlyb{t} \subseteq
\Kt\curlyb{t}$ is finite (details to all statements in this paragraph may be found
in \cite[\textsection 2.3.1]{morrow-explicit-approach-to-residues}).

%From the point of view of the theory of nonarchimedean functional
%analysis, $K$-vector spaces, tudy of two-dimensional fields $K \subset F$ as $K$-vector
%spaces there are two 
%
%In any case, the field $\Kt$ will play a helpful role in our constructions, and so will the field $\Kt \curlyb{t}$ in the mixed characteristic case.

%Regardless of $\car \overline{F}$, the existence of a field inclusion $K
%\subset F$ forces a certain compatibility between rank-two valuations
%of $F$ and the discrete valuation of $K$. Namely, we have an inclusion of
%abelian groups $K^\times \subset F^\times$. The structure of these
%abelian groups is well known, and implies that one of the components of a
%rank-two valuation of $F$ restricts to the discrete valuation of $K$ and
%the other one restricts trivially.

Since the residue field of a two-dimensional local field is a local
field, we can use the discrete valuation at the residue level to define a
rank-two valuation on $F$ as follows: order $\ZZ \oplus \ZZ$ by
$(n_1,m_1) < (n_2,m_2)$ if and only if $n_1 < n_2$ or $n_1=n_2$ and $m_1
< m_2$ and consider, after choosing a uniformizer $\pi \in F$,
\begin{equation*}
  (v_F, v_\pi) : F^\times \to \ZZ \oplus \ZZ
\end{equation*}
with $v_\pi(x) := v_{\overline{F}}\left( \overline{x\pi^{-v_F(x)}}
\right)$. The valuation ring
\begin{equation*}
  \mathrm{O}_F = \left\{ x \in F;\; (v_F(x), v_\pi(x)) \geq (0,0) \right\} 
\end{equation*}
does not depend on the choice of uniformizer \cite[\textsection 1]{ihlf}.

\begin{exa}
  Consider $K = \QQ_p \subset \QQ_p\curlyb{t} = F$. For the choice of
  uniformizer $p$ for $v_F$, the associated rank-two valuation of $F$ is
  \begin{equation*}
    \left( v_1, v_2 \right): F^\times \to \ZZ \oplus \ZZ, \quad x =
    \sum_{i \in \ZZ} x_i t^i \mapsto \left( \inf_{i\in \ZZ} v_p\left( x_i
    \right), \inf\left\{ i;\; x_i \notin p^{v_1(x)+1}\ZZ_p \right\} \right).
  \end{equation*}
  The restriction of $v_1$ to $K$ is $v_p$, while $v_2$ restricts
  trivially. The rank-two ring of integers is
  \begin{equation*}
    \mathrm{O}_F = \left\{ \sum_{i \in \ZZ} x_i t^i \in F;\; x_i \in
    p\ZZ_p \text{ for } i < 0 \text{ and } x_i \in \ZZ_p \text{ for } i
    \geq 0 \right\}.
  \end{equation*}
\end{exa}

\begin{exa}
  Consider $K = \QQ_p \subset \QQ_p \roundb{t} = F$. In such case, the
  rank-two valuation of $F$ associated to the uniformizer $t$ for
  $v_F$ is
  \begin{equation*}
    \left( v_1, v_2 \right): F^\times \to \ZZ \oplus \ZZ, \quad \sum_{i \geq i_0} a_i t^i \mapsto(i_0, v_p(a_{i_0})),
  \end{equation*}
  where we suppose that $a_{i_0}$ is the first nonzero coefficient in the
  power series. The restriction of $v_1$ to $K$ is trivial while the
  restriction of $v_2$ to $K$ is $v_p$. In this case we have
  $\mathrm{O}_F = \ZZ_p + t \QQ_p\squareb{t}$.
\end{exa}

\begin{rmk}
There are two particular local fields which play a
very distinguished role when these objects are to be studied from a
functional analytic point of view. Those are the fields
$K\roundb{t}$ and $K\curlyb{t}$. As we will see, most topological
properties which hold in these particular cases will hold in general
after taking restrictions of scalars or a base change over a finite
extension which topologically is equivalent to taking a finite cartesian
product. It is for this reason that we will work from now on with these two particular
examples of two-dimensional local fields. We will explain how our results
extend to the general case in \textsection \ref{sec:generalcase}.
\end{rmk}

\begin{nota}
  When working with the two-dimensional local fields $F = K\curlyb{t}$ or
  $F = K\roundb{t}$, for any collection $\left\{ A_i \right\}_{i \in \ZZ}$ of
  subsets of $K$, we will denote
  \begin{equation*}
    \sum_{i \in \ZZ} A_i t^i = \left\{ \sum_{i} x_i t^i \in F;\; x_i \in
    A_i \text{ for all } i\in \ZZ \right\}.
  \end{equation*}

  We will also denote $\OO_{K\curlyb{t}} = \OO\curlyb{t}$. After all,
  this ring consists of all power series in $K\curlyb{t}$ all of whose
  coefficients lie in $\OO$.
\end{nota}

\section{Higher topologies are locally convex}
\label{sec:highertopislocconv}

%Let $F$ be a two-dimensional local field, and $K \subset F$ a subfield of $F$ which is a one-dimensional local field. %Given a two-dimensional local field, the existence of such a one-dimensional local field contained in it is granted by structure theory, and when our fields arise from the geometry of arithmetic and algebraic surfaces such a choice is natural.

In this section we will explain how the higher topology on $K\roundb{t}$ and
$K\curlyb{t}$ is a locally convex topology. Higher topologies for
two-dimensional local fields were first introduced in \cite{parshin-lCFT}
in the study of two-dimensional class field theory in positive
characteristics. The general construction is available at
\cite{madunts-zhukov-topology-hlfs}, while \cite[\textsection 1]{ihlf}
contains an accessible survey on the topic.

%The structure of $F$ changes depending on its characteristic and the characteristic of $\overline{F}$, and the way to define a higher topology changes accordingly. 

We are forced to study both cases separately.

\subsection{Equal characteristic}
\label{sec:equicarcase}

%\begin{exa}
%  If $\car F = \car \overline{F} = 0$, then $K = \overline{F}$ is a finite extension of $\mathbb{Q}_p$, with $p$ being the characteristic of the last residue field of $F$. If, otherwise, $\car F = \car \overline{F} = p$, then $F \cong k\roundb{u}\roundb{t}$ with $k$ a finite field of characteristic $p$.
%\end{exa}

%As describing higher topology is much easier once an isomorphism $F \cong \Kt \roundb{t}$ has been fixed, we will assume until further notice that $K$ is algebraically closed in $F$ (and hence $K = \Kt$), and that a parameter $t$ has been fixed. 

The higher topology on $K\roundb{t}$ is defined as follows. Let $\left\{ U_i \right\}_{i \in \ZZ}$ be a collection of open neighbourhoods of zero in $K$ such that, if $i$ is large enough, $U_i = K$. Then define
\begin{equation}
  \label{eqn:highertopopennhoodofzero}
  \mathcal{U} = \sum_{i \in \ZZ} U_i t^i. % = \left\{ \sum_{i \gg -\infty} x_i t^i \in F;\; x_i \in U_i \right\}.
\end{equation}

The collection of sets of the form $\mathcal{U}$ defines the set of
neighbourhoods of zero of a group topology on $K\roundb{t}$
\cite[\textsection 1]{madunts-zhukov-topology-hlfs}.

\begin{prop}
  \label{prop:locconvequicase}
  The higher topology on $K\roundb{t}$ defines the structure of a locally convex $K$-vector space.
\end{prop}

\begin{proof}
  As $K$ is a local field, the collection of open neighbourhoods of zero
  admits a collection of open subgroups as a filter, that is: the basis
  of neighbourhoods of zero for the topology is generated by the sets of the form
  \begin{equation*}
    \pp^n = \left\{ a \in K;\; v_K(a) \geq n \right\},
    %B_\varepsilon = \left\{ a \in K; \; \vert a \vert_K \leq \varepsilon \right\}.
  \end{equation*}
  where $n \in \ZZ \cup \left\{ -\infty \right\}$. These closed balls are not only subgroups, but $\OO$-fractional ideals. This in particular means that the sets of the form
  \begin{equation}
    \label{eqn:equicarlattices}
    \Lambda = \sum_{i \in \ZZ} \pp^{n_i} t^i \subseteq K\roundb{t},
  \end{equation}
  where $n_i = -\infty$ for large enough $i$, generate the higher topology. Moreover, they are not only additive subgroups, but also $\OO$-modules. 
  
  If $x = \sum_{i \geq i_0} x_i t^i \in K\roundb{t}$ is an arbitrary element, and $i_1$ is such that $n_i = -\infty$ for all $i > i_1$ then we have the possibilities:
  \begin{enumerate}
    \item $i_0 > i_1$, in which case $x \in \Lambda$.
    \item $i_0 \leq i_1$. In such case, let 
      \begin{equation*}
	n = \max \left( \max_{i_0 \leq i \leq i_1} n_i, 0 \right) .
      \end{equation*}
      Then $\pi^n \in \OO$ satisfies $\pi^n x \in \Lambda$.
  \end{enumerate}
  Thus, $\Lambda$ is a lattice and the higher topology is locally convex.
\end{proof}

As a consequence of the previous proposition, it is possible to describe the higher topology in terms of seminorms.

\begin{corollary}
  \label{cor:seminormsequi}
  For any
  sequence $\left( n_i \right)_{i \in \ZZ} \subset \ZZ \cup \left\{
  -\infty \right\}$ such that there is an integer $k$ satisfying
  $n_i = -\infty$ for all $i > k$, define
  \begin{equation}
    \label{eqn:highertopinseminorms}
    \| \cdot \|: K\roundb{t} \to \mathbb{R},\quad \sum_{i \gg -\infty}
    x_i t^i \mapsto \max_{i \leq k} |x_i| q^{n_i}.
  \end{equation}
  Then $\| \cdot \|$ is a seminorm on $K\roundb{t}$ and the higher
  topology on $K\roundb{t}$ is the locally convex topology defined by the
  family of seminorms given by (\ref{eqn:highertopinseminorms}) as 
  $( n_i )_{i \in \ZZ}$ varies over all sequences specified above.
\end{corollary}

\begin{proof}
  This result is a consequence of Proposition \ref{prop:locconvequicase} and of the fact that the gauge seminorm attached to a lattice of the form
  \begin{equation*}
    \Lambda = \sum_{i \in \ZZ} \pp^{n_i} t^i
  \end{equation*}
  with $n_i = \infty$ for all $i > k$ is precisely the one given by
  (\ref{eqn:highertopinseminorms}). In order to see that, let $x =
  \sum_{i \geq i_0}  x_i t^i \in K\roundb{t}$ and $a \in K$. We have that
  $x \in a\Lambda$ if and only if $x_i \in a \pp^{n_i}$ for every
  $i \geq i_0$. This is the case if and only if we have
  \begin{equation*}
    |x_i|q^{n_i} \leq |a|
  \end{equation*} 
  for all $i \geq i_0$.
  The infimum value of $|a|$ for which the above inequality holds is
  precisely the supremum of the values of $|x|q^{n_i}$ for $i \geq i_0$.
\end{proof}

The seminorm $\| \cdot \|$ from
the previous corollary is associated to and does depend on the choice of the sequence
$\left( n_i \right)_{i \in \ZZ}$. If we have chosen notation not to
reflect this fact, it is
in hope that a lighter notation will simplify reading and that the
sequence of integers defining $\| \cdot \|$, when needed, will be clear from the
context.

\begin{rmk}
As $F$ is a field, it is worth asking ourselves whether the seminorm
(\ref{eqn:highertopinseminorms}) is multiplicative. It is very easy to check that for $i, j \in \ZZ$,
\begin{equation*}
  \| t^i \| \cdot \| t^j \| = q^{n_i + n_j},
\end{equation*}
while
\begin{equation*}
  \|t^{i+j}\| = q^{n_{i+j}}.
\end{equation*}
These two values need not coincide in general.
\end{rmk}

The field of Laurent series $K\roundb{t}$ has been considered previously
from the point of view of the theory of locally convex spaces in
the following manner. The ring of Taylor series $K\squareb{t}$ is 
isomorphic to $K^\mathbb{N}$ as a $K$-vector space, and thus might be
equipped with the product topology of countably many copies of $K$. 
Moreover, we have
\begin{equation}
  K\roundb{t} = \cup_{i \in \ZZ} t^{i} K\squareb{t},
  \label{eqn:highertopequicarisstrict}
\end{equation}
with $t^{i}K\squareb{t} \cong K^\mathbb{N}$. Therefore, we may topologize $K\roundb{t}$ as a strict inductive limit.

In the result below, we explain how the higher topology on
$K\roundb{t}$ agrees with this description. We will immediately
deduce most of the analytic properties of $K\roundb{t}$ from this
result.

\begin{prop}
  \label{prop:toponKsquarebisprodtop}
  The higher topology on $K\roundb{t}$ agrees with the strict inductive
  limit topology given by (\ref{eqn:highertopequicarisstrict}).
\end{prop}

\begin{proof}
  The open lattices for the product topology on $K^\mathbb{N}$ are
  exactly the ones of the form
  \begin{equation*}
    \prod_{i \in I} \Lambda_i \times \prod_{i \notin I} K,
  \end{equation*}
  where $I$ is a finite subset of $\mathbb{N}$ and $\Lambda_i$
  are open lattices in $K$, that is, integer powers of $\pp$. This
  description agrees with the description of the open lattices in
  $K\squareb{t}$ for the subspace topology induced by the higher topology.

  Further, if $\Lambda = \sum_{i \in \ZZ} \pp^{n_i} t^i$ is an open
  lattice for the higher topology on $K\roundb{t}$, for
  any $j \in \ZZ$, we have that $\Lambda \cap t^{j}K\squareb{t} =
  \sum_{i \geq j} \pp^{n_i} t^i$ is an open lattice for the product
  topology on $t^j K\squareb{t} \cong K^\mathbb{N}$.

  Finally, any open lattice for the strict inductive limit
  $\bigcup_{j \in \ZZ} t^j K\squareb{t}$ is given by a collection of open
  lattices $\Lambda_j \subseteq t^j K\squareb{t}$ for each $j \in \ZZ$. These
  are of the form $\Lambda_j = \sum_{i \geq j} \pp^{n_{i,j}} t^i$ for some sequence
  $(n_{i,j})_{i \geq j} \subset \ZZ \cup \left\{ -\infty \right\}$ for which there is an index $k_i \geq i$ such
  that $n_{i,j} = -\infty$ for all $j \geq k_i$. The fact that the
  inductive limit is strict amounts to the following: if $i_1 < i_2$ then we
  have $n_{i_1, j} = n_{i_2, j}$ for every $j \geq i_2$ and, in
  particular, $k_{i_1} = k_{i_2}$. Altogether, this determines a sequence
  $(n_i)_{i \in \ZZ} \subset \ZZ \cup \left\{ -\infty \right\}$ and an
  index $k \in \ZZ$ such that $n_i = -\infty$ for every $i \geq k$. Under
  the identification $K\roundb{t} = \bigcup_{j \in \ZZ} t^j
  K\squareb{t}$, the lattice associated to $(\Lambda_j)_{j \in \ZZ}$ is
  $\Lambda = \sum_{i \in \ZZ} \pp^{n_i} t^i$, which is open for the
  higher topology.
\end{proof}

\begin{rmk}
  \label{rmk:directlimequicar}
  The higher topology on $K\roundb{t}$ also admits the following
  description as an inductive limit. For each $i \in
  \ZZ$ and $j \geq i$, $t^i K[t] / t^j K[t]$ is a finite dimensional
  $K$-vector space and we endow it with its unique Hausdorff locally
  convex topology. The field of Laurent series might be constructed as
  \begin{equation*}
    K\roundb{t} = \varinjlim_{i \in \ZZ} \varprojlim_{j \geq i} t^i K[t] / t^j K[t];
  \end{equation*}
  the higher topology on it agrees with the one obtained by endowing the
  direct and inverse limits in the above expression with the
  corresponding direct and inverse limit locally convex topologies. The
  proof of this statement is a restatement of Proposition
  \ref{prop:toponKsquarebisprodtop}.
%
%  After the previous proposition, we can describe the higher topology on
%  $K\roundb{t}$ as follows: consider the product topology on
%  $t^{-i} K \squareb{t} \cong K^\mathbb{N}$ and topologize
%  \begin{equation*}
%    K\roundb{t} = \cup_{i \in \mathbb{N}} t^{-i} K\squareb{t}
%  \end{equation*}
%  using the strict inductive limit topology (Definition
%  \ref{df:strictdirectlim}). This in particular implies that the higher
%  topology on $K\roundb{t}$ may be described by taking simpler
%  seminorms. The present formulation will be useful to us as it
%  is common for both $K\roundb{t}$ and $K\curlyb{t}$.
\end{rmk}

\begin{corollary}
  \label{cor:propsofequicar}
  $K\roundb{t}$ is an LF-space.
  In particular, it is complete, bornological, barrelled, reflexive and nuclear.
\end{corollary}

\begin{proof}
  After Proposition \ref{prop:toponKsquarebisprodtop}, in order to see
  that $K\roundb{t}$ is an LF-space it suffices to check that the locally
  convex space $K^\mathbb{N}$ endowed with the product topology is a
  Fr\'echet space. This follows from the fact that $K$ itself is
  Fr\'echet and that a countable product of Fr\'echet spaces is a
  Fr\'echet space
  \cite[Corollary
  3.5.7]{perez-garcia-schikof-locally-convex-spaces-nonarchimedean-valued-fields}.

  Completeness follows from being a strict inductive limit of complete
  spaces \cite[Lemma 7.9]{schneider-non-archimedean-functional-analysis}.
  Being bornological and barrelled follow from
  \cite[Proposition 8.2]{schneider-non-archimedean-functional-analysis},
  reflexivity follows from
  \cite[Corollary
  7.4.23]{perez-garcia-schikof-locally-convex-spaces-nonarchimedean-valued-fields}
  and nuclearity follows from Proposition \ref{prop:propsofnuclearspaces}
\end{proof}

\subsection{Mixed characteristic}

%Assume from now on that $F$ is a mixed characteristic two-dimensional local field. In order to define a higher topology on $F$, one defines a lifting map
%\begin{equation*}
%  h : \overline{F} \to \OO_F^\times \cup \left\{ 0 \right\},
%\end{equation*}
%that is, a set-theoretic section of the residue map $\OO_F \to \overline{F}$ with certain special properties, which generalises the choice of a coefficient subfield in the equal characteristic case.

The higher topology on $K\curlyb{t}$ may be described as follows.

Let $\left\{ V_i \right\}_{i \in \ZZ}$ be a sequence of open neighbourhoods of zero in $K$ such that
\begin{enumerate}
  \item There is $c \in \ZZ$ such that $\pp^c \subseteq V_i$ for every $i \in \mathbb{Z}$.
  \item For every $l \in \ZZ$ there is an index $i_0 \in \ZZ$ such that $\pp^l \subseteq V_i$ for every $i \geq i_0$.
\end{enumerate}
Then define
\begin{equation}
  \label{eqn:mixedchartop}
  \mathcal{V} = \sum_{i \in \ZZ} V_i t^i \subset K\curlyb{t}. % = \left\{ \sum_{i = -\infty}^\infty x_i t^i \in F;\; x_i \in V_i \right\}.
\end{equation}
The higher topology on $K\curlyb{t}$ is the group topology defined by taking the sets of the form $\mathcal{V}$ as the collection of open neighbourhoods of zero \cite[\textsection 1]{madunts-zhukov-topology-hlfs}.

Again, as $K$ is a local field, the collection of neighbourhoods of zero admits the collection of open subgroups as a filter. These are not only subgroups but $\OO$-fractional ideals, namely the integer powers of the prime ideal $\pp$. 

\begin{prop}
  \label{prop:locconvmixedcase}
  Let $(n_i)_{i \in \ZZ} \subset \ZZ \cup \left\{ -\infty \right\}$ be a
  sequence restricted to the conditions:
  \begin{enumerate}
    \item There is $c \in \ZZ$ such that $n_i \leq c$ for every $i$.
    \item For every $l \in \ZZ$ there is an index $i_0 \in \ZZ$ such that
      $n_i \leq l$ for every $i \geq i_0$.
  \end{enumerate}
  The set 
  \begin{equation}
    \label{eqn:latticemixedchar}
    \Lambda = \sum_{i \in \ZZ} \pp^{n_i}t^i
  \end{equation}
  is an $\OO$-lattice. The sets of the form (\ref{eqn:latticemixedchar})
  generate the higher topology on $K\curlyb{t}$, which is locally convex.
\end{prop}

Condition (ii) is equivalent, by definition of limit of a sequence, to
having $n_i \to -\infty$ as $i \to \infty$; we will phrase it this way
in the future.

\begin{proof}
  It is clear that $\Lambda$ is an $\OO$-module, and that the conditions
  imposed on the indices $n_i$ imply that it is a basic neighbourhood of zero for the higher topology.

%  Assume that $x = \sum_{i = -\infty}^\infty x_i t^i \in \Lambda$. This means that $x_i \in \pp^{n_i}$ for every $i \in \ZZ$. As these fractional ideals are $\OO$-modules, and multiplication by elements in $\OO$ is done componentwise, $\Lambda$ is an $\OO$-submodule.

Given an arbitrary element $x = \sum_{i=-\infty}^\infty x_i t^i \in F$,
we must show the existence of an element $a \in K^\times$ such that $ax
\in \Lambda$. Indeed, a power of the
uniformizer does the trick: we have that $\pi^n x \in \Lambda$ if and
only if $\pi^n x_i \in \pp^{n_i}$ for every $i \in \ZZ$, and this is true if and only if
\begin{equation*}
  n + v_K(x_i) \geq n_i
\end{equation*} 
for all $i \in \ZZ$. In other words, such an $n$ exists if and only if the difference
\begin{equation*}
  n_i - v_K(x_i)
\end{equation*}
cannot be arbitrarily large. But on one hand there is an integer $c$ that bounds the $n_i$ from above, and on the other hand the values $v_K(x_i)$ are bounded below by $v_F(x)$. We may take $n = c - v_F(x)$.

Because the integer powers of $\pp$ generate the basis of neighbourhoods of zero of the topology on $K$, the lattices of the form (\ref{eqn:latticemixedchar}) generate the higher topology. In particular, the higher topology on $K\curlyb{t}$ is locally convex.
\end{proof}

We wish to point out that condition (ii) for the sequence $(n_i)_{i\in
\ZZ}$ has not been used in the proof. Indeed, such a condition may be
suppressed and we would still obtain a locally convex topology on
$K\curlyb{t}$, if only finer: see Remark
\ref{rmk:mixedcharsuppressconditionii} for a description of the topology
obtained in such case.

Once we know that the higher topology is locally convex, we can describe it in terms of seminorms.

\begin{corollary}
  \label{cor:seminormsmixedstdcase}
  For any sequence $\left( n_i \right)_{i \in \ZZ} \subset \mathbb{Z}
  \cup \left\{ -\infty \right\}$ satisfying the conditions:
  \begin{enumerate}
    \item there is $c \in \ZZ$ such that $n_i \leq c$ for all $i \in
      \ZZ$,
    \item $n_i \to -\infty$ as $i \to \infty$,
  \end{enumerate}
  consider the seminorm
  \begin{equation}
    \label{eqn:highertopinseminormsmixed}
    \| \cdot \|: K\curlyb{t} \to \mathbb{R},\quad \sum_{i \in \ZZ} x_i
    t^i \mapsto \sup_{i \in \ZZ} |x_i| q^{n_i}.
  \end{equation}
  The higher topology on $K\curlyb{t}$ is the locally convex topology
  generated by the family of seminorms defined by
  (\ref{eqn:highertopinseminormsmixed}), as $(n_i)_{i \in \ZZ}$ varies
  over the sequences specified above.
\end{corollary}

\begin{proof}
  The gauge seminorm associated to the lattice $\Lambda$ is
  (\ref{eqn:highertopinseminormsmixed}). The argument is the same
  as the proof of Corollary \ref{cor:seminormsequi} and we
  omit it.
\end{proof}

The seminorms in Corollary \ref{cor:seminormsmixedstdcase} are well defined because they arise as gauge seminorms attached to lattices. If we forget this fact for a moment, let us examine the values $|x_i| q^{n_i}$.

On one hand, when $i$ tends to $-\infty$, the values $|x_i|$ tend to
zero while the values $q^{n_i}$ stay bounded. On the other hand,
when $i$ tends to $+\infty$ the values $|x_i|$ stay bounded 
and $q^{n_i}$ tends to zero. In conclusion, the values $|x_i| q^{n_i}$ are all positive and tend to
zero when $|i| \to +\infty$; this implies the existence of their
supremum.

Just like in the equal characteristic case, a defining seminorm $\| \cdot \|$ is not multiplicative, for the same reason.

A mixed characteristic two-dimensional local field cannot be viewed as
a direct limit in a category of locally convex $K$-vector spaces in the
fashion of Remark \ref{rmk:directlimequicar}. However,
such an approach is valid from an algebraic point of view in a category
of $\OO$-modules.

\subsection{First topological properties}
\label{sec:firsttopoprop}

For starters, let us recall a few well-known properties of higher
topologies. A two-dimensional local field $K \hookrightarrow F$ endowed
with a higher topology is a Hausdorff topological group
\cite[Theorem 1.1.i and Proposition 1.2]{madunts-zhukov-topology-hlfs}. Moreover,
multiplication by a fixed nonzero element defines a homeomorphism $F \to
F$ \cite[Theorem 1.1.ii]{madunts-zhukov-topology-hlfs} and the residue
map $\OO_F \to \overline{F}$ is open when $\OO_F$ is given the subspace
topology and the local field $\overline{F}$ is endowed with its usual
complete discrete valuation topology \cite[Proposition 3.6.(v)]{camara-top-rat-pts-hlfs-arxiv}.

\begin{rmk}
  In order to show that $K\roundb{t}$ or $K\curlyb{t}$ is Hausdorff, it
  suffices to show that given a nonzero element $x$, there is an admissible 
  seminorm $\| \cdot \|$ for which $\| x \| \neq 0$. This is obvious.
\end{rmk}

Multiplication $\mu: F \times F \to F$ fails to be continuous when the product
topology is considered on the left hand side \cite[\textsection 1.3.2]{ihlf}. However, $\mu$ is
separately continuous as explained above.

Another well known fact about higher topologies is that no basis of open
neighbourhoods of zero is countable \cite[\textsection 1.3.2]{ihlf}. In other words, these topologies do
not satisfy the first countability axiom. This implies that the set of 
seminorms defining the higher topology is
uncountable. From the point of view of functional analysis, this shows
that two-dimensional local fields are not Fr\'echet spaces.

\begin{df}
  We will call seminorms of the form (\ref{eqn:highertopinseminorms})
  in the equal characteristic case and
  (\ref{eqn:highertopinseminormsmixed}) in the mixed characteristic case
  {\it admissible}. 
\end{df}

  In both cases, admissible seminorms are attached to a
  sequence $(n_i)_{i \in \ZZ} \subset \ZZ \cup \left\{ -\infty \right\}$,
  subject to different conditions, but satisfying the formula
  \begin{equation*}
    \left\| \sum_i x_i t^i \right\| = \sup_i |x_i| q^{n_i};
  \end{equation*}
  the reasons why this formula is valid differ in each case.

\begin{rmk}
  \label{rmk:powerseriesconverge}
  Power series expressions of the form $x = \sum_i x_i t^i$ define
  convergent series with respect to the higher topology, in the sense
  that the net of partial sums $\left( \sum_{i \leq n} x_i t^i
  \right)_{n \in \ZZ}$ converges to $x$. If we let $S_n = \sum_{i \leq n}
  x_i t^i$ and $\| \cdot \|$ be any admissible seminorm, then
  \begin{equation*}
    \| x - S_n \| = \left\| \sum_{i > n} x_i t^i \right\|
  \end{equation*}
  may be shown to be arbitrarily small if $n$ is large enough.
\end{rmk}

%It is obvious that, being a locally convex $K$-vector space, $F$ is also a topological $K$-vector space.

Another well-known fact is that rings of integers $K\squareb{t}$ and
$\OO\curlyb{t}$ are closed but not open. In the first case, consider the
set of open (and closed) lattices
\begin{equation*}
  \Lambda_n = \sum_{i \leq 0} \pp^n t^i + K\squareb{t},\quad n \geq 0
\end{equation*}
to find that $K\squareb{t} = \bigcap_{n \geq 0} \Lambda_n$ is closed. In the
second case, consider the open (and closed) lattices:
\begin{equation*}
  \Lambda_n = \sum_{i < n} Kt^i + \OO t^n + \sum_{i > n} Kt^i,\quad n \in
  \ZZ
\end{equation*}
and obtain that $\OO\curlyb{t} = \bigcap_{n \in \ZZ} \Lambda_n$ is
closed. In order to see that these rings are not open, it is enough to
say that they do not contain any open lattice.

A very similar argument shows that the rank-two rings of integers
$\OO + t K\squareb{t}$ and $\sum_{i < 0} \pp t^i + \sum_{i \geq 0} \OO
t^i$ are closed but not open.

After the previous remark, we get the following result.

\begin{prop}
  \label{prop:curlybisnotbarrelled}
  The field $K\curlyb{t}$ is not barrelled.
\end{prop}

\begin{proof}
  The ring of integers $\OO\curlyb{t}$ is a lattice which is closed but
  not open. 
\end{proof}

\section{Bounded sets and bornology}
\label{sec:bornology}

Let us describe the nature of bounded subsets of $K\roundb{t}$ and
$K\curlyb{t}$. We will supply a description of a basis for the
Von-Neumann bornology of these fields.

\begin{exa}
  \label{exa:Kisunbounded}
  Let $\| \cdot \|$ be an admissible seminorm, attached to the
  sequence $(n_i)_{i \in \ZZ}$. The values of $\| \cdot
  \|$ on $\OO$ only depend on $n_0$. If $n_0 = -\infty$ then the
  restriction of $\| \cdot \|$ to $\OO$ is identically zero. Otherwise,
  for any $x \in \OO$ we have $\| x \| \leq q^{n_0}$ and therefore $\OO$ is bounded.
  
  Similarly, if $n_0 > -\infty$, we may find elements $x \in K$ making the value $|x| q^{n_0}$ arbitrarily large. Hence, $K$ is unbounded.
\end{exa}

\begin{prop}
  \label{prop:basicboundedsetequicar}
  For any sequence $(k_i)_{i \in \ZZ} \subset \ZZ \cup \left\{ \infty
  \right\}$ such that there is an index $i_0 \in \ZZ$ for which
  $k_i = \infty$ for every $i < i_0$, consider the $\OO$-submodule of
  $K\roundb{t}$ given by
  \begin{equation}
    B = \sum_{i \in \ZZ} \pp^{k_i} t^i.
    \label{eqn:basicboundedsetequicar}
  \end{equation}
  The bornology of $K\roundb{t}$ admits as a basis the collection of
  $\OO$-submodules given by (\ref{eqn:basicboundedsetequicar}) as
  $(k_i)_{i \in \ZZ}$ varies over the sequences specified above.
\end{prop}

\begin{proof}
  First, the $\OO$-submodule $B$ given by (\ref{eqn:basicboundedsetequicar}) is
  bounded: suppose that $\| \cdot \|$ is an admissible seminorm on
  $K\roundb{t}$ given by the sequence $(n_i)_{i \in \ZZ}$ and that $k$ is the index for which $n_i = -\infty$ for every $i > k$.

  If $k < i_0$, then the restriction of $\| \cdot \|$ to $B$ is
  identically zero. Otherwise, for $x = \sum_{i \geq i_0} x_i t^i \in B$,
  \begin{equation*}
    \|x\| = \max_{i_0 \leq i \leq k} |x_i| q^{n_i} \leq \max_{i_0 \leq i \leq k} q^{n_i - k_i},
  \end{equation*}
  and the bound is uniform for $x \in B$ once $\| \cdot \|$ has been fixed.

  Next, we study general bounded sets. From Example
  \ref{exa:Kisunbounded} we deduce that if a subset of $K\roundb{t}$
  contains elements for which one coefficient can be arbitrarily large,
  then the subset is unbounded in $F$. Therefore, any bounded subset of
  $K\roundb{t}$ is included in a subset of the form
  \begin{equation*}
    \sum_{i \in \ZZ} \pp^{k_i} t^i,\quad k_i \in \ZZ \cup \left\{
    \infty \right\}.
  \end{equation*}
  In order to prove our claim, it is enough to show that the indices
  $k_i \in \ZZ \cup \left\{ \infty \right\}$ may be taken to be equal to
  $\infty$ for all small enough $i$.

  We will show the contrapositive: a subset $D \subset
  K\roundb{t}$ cannot be bounded as soon as there is a decreasing sequence of indices $(i_j)_{j \geq 0}
  \in \ZZ_{<0}$ satisfying that for every $j \geq 0$ there
  is an element $\xi_{i_j} \in D$ with a nonzero coefficient in degree
  $i_j$, which we denote $x_{i_j} \in K$. 
  
  For, if such is the case, let 
  \begin{equation*}
    n_i = \left\{ 
    \begin{array}{ll}
      -\infty, & i \neq i_j \text{ for any } j \geq 0 \text{ or } i > 0, \\
      -i_j + v_K(x_{i_j}), & i = i_j \text{ for some } j \geq 0;
    \end{array}
    \right.
  \end{equation*}
  and consider the associated admissible seminorm $\| \cdot \|$ on
  $K\roundb{t}$. We have
  \begin{equation*}
    \| \xi_{i_j} \| \geq |x_{i_j}| q^{n_{i_j}} = q^{-i_j}
  \end{equation*}
  for every $j \geq 0$, and this shows that $D$ is not bounded.
\end{proof}

\begin{corollary}
  If $\| \cdot \|: K\roundb{t} \to \mathbb{R}$ is a seminorm
  which is bounded on bounded sets, then there is an index $i_0 \in \ZZ$
  such that $\|t^i\| = 0$ for all $i \geq i_0$.
\end{corollary}

\begin{proof}
  Suppose that for every $i_0 \in \ZZ$ there is an $i \geq i_0$ such that
  $\|t^i\| \neq 0$. If $i$ is such that $\|t^i\| > 0$, take $k_i \in \ZZ$ such that
  \begin{equation*}
    q^{-k_i}\|t^i\| \geq q^i.
  \end{equation*}
  If $i$ is such that $\|t^i\| = 0$, take $k_i = 0$. By Proposition \ref{prop:basicboundedsetequicar}, the set
  \begin{equation*}
    B = \sum_{i \geq 0} \pp^{k_i} t^i
  \end{equation*}
  is bounded. Let $x_j = \pi_K^{k_j} t^j$ for every $j \geq 0$. We have
  that $\|x_j\| = q^{-k_j}\|t^{j}\|$. Our hypothesis implies that the sequence of real numbers $\left(
  \|x_j\| \right)_{j \geq 0}$ is unbounded, and therefore $\| \cdot \|$ is not bounded on $B$.
\end{proof}

\begin{prop}
  \label{prop:basicboundedsetmixedcar}
  Consider a sequence $(k_i)_{i \in \ZZ}\subset \ZZ \cup
  \left\{ \infty \right\}$ which is bounded below. The bornology of
  $K\curlyb{t}$ admits the $\OO$-submodules of the form
  \begin{equation}
    B = \sum_{i \in \ZZ} \pp^{k_i} t^i
    \label{eqn:basicboundedsetmixedcar}
  \end{equation}
  as a basis.
\end{prop}

\begin{proof}
  First, let us show that $B$ is bounded. Assume all the $k_i$
  in (\ref{eqn:basicboundedsetmixedcar}) are bounded below by some
  integer $d$. Let $\| \cdot \|$ be an admissible seminorm on
  $K\curlyb{t}$ defined by a sequence $(n_i)_{i \in \ZZ} \subset \ZZ \cup
  \left\{ -\infty \right\}$. In particular, there is an integer $c$ such
  that $n_i \leq c$ for every $i \in \ZZ$.
  
  Then, if $\sum x_i t^i \in B$, we have that
  \begin{equation*}
    \left\| \sum x_i t^i \right\| = \sup_i |x_i| q^{n_i} \leq q^{c-d},
  \end{equation*}
  and the bound is uniform on $B$ once $\| \cdot \|$ has been fixed.

  Next, we study general bounded sets. Again, from Example
  \ref{exa:Kisunbounded} we may deduce that a subset of $K\curlyb{t}$
  which contains elements with arbitrarily large coefficients cannot be
  bounded. Therefore, any bounded subset of $K\curlyb{t}$ is contained in
  a set of the form
  \begin{equation*}
    \sum_{i \in \ZZ} \pp^{k_i} t^i, k_i \in \ZZ \cup \left\{ \infty
    \right\}.
  \end{equation*}
  In order to prove our claim, it is enough to show that the indices
  $k_i$ may be taken to be bounded below.

  Suppose that $D \subset K\curlyb{t}$ is not contained
  in a set of the form (\ref{eqn:basicboundedsetmixedcar}). Then it must
  contain elements with arbitrarily large coefficients. More precisely,
  at least one of the following must happen:
  \begin{enumerate}[1.]
    \item There is a decreasing sequence $(i_j)_{j \geq 0} \subset
      \ZZ_{<0}$ and a sequence $(\xi_{i_j})_{j \geq 0} \subset D$ such
      that, if $x_{i_j} \in K$ denotes the coefficient in degree
      $i_j$ of $\xi_{i_j}$, we have $|x_{i_j}| \to \infty$ as $j \to
      \infty$.
    \item There is an increasing sequence $(i_j)_{j \geq 0} \subset
      \ZZ_{\geq 0}$ and a sequence $(\xi_{i_j})_{j \geq 0} \subset D$ such
      that, if $x_{i_j} \in K$ denotes the coefficient in degree
      $i_j$ of $\xi_{i_j}$, we have $|x_{i_j}| \to \infty$ as $j \to
      \infty$.
  \end{enumerate}

  If condition 1 holds, consider the admissible seminorm $\| \cdot \|$ associated
  to the sequence
  \begin{equation*}
    n_i = \left\{ 
    \begin{array}{ll}
      0,&\text{if } i \leq 0,\\
      -\infty,&\text{if } i > 0.
    \end{array}
    \right.
  \end{equation*}
  We have that $\|\xi_{i_j}\| \geq |x_{i_j}|$ for all $j \geq 0$ and this
  implies that $D$ cannot be bounded.

  If condition 1 does not hold, then condition 2 must hold. In such case, define
  \begin{equation*}
    n_{i_j}= \left\{ 
    \begin{array}{ll}
      \frac{v_K(x_{i_j})-1}{2}&\text{if } v_K(x_{i_j}) \text{ is odd},\\
      \frac{v_K(x_{i_j})}{2},&\text{if } v_K(x_{i_j}) \text{ is even}.
    \end{array}
    \right.
  \end{equation*}
  Furthermore, let $n_i = -\infty$ for any $i < 0$ and $n_l =
  n_{i_j}$ for any index $l$ such that $i_j \leq l < i_{j+1}$. With such choices, the following three facts hold:
  \begin{enumerate}
    \item The sequence $(n_{i_j} - v_K(x_{i_j}))_{j \geq 0}$ tends to infinity.
    \item The sequence $( n_i )_{i \in \ZZ}$ is bounded above.
    \item For any $l \in \ZZ$, there is an index $i_0$ such that $n_i \leq l$ for all $i \geq i_0$.
  \end{enumerate}
  After (ii) and (iii), let $\| \cdot \|$ be the admissible seminorm
  associated to $(n_i)_{i \in \ZZ}$. We have, for every $j \geq 0$, $\|
  \xi_{i_j} \| \geq |x_{i_j}|q^{n_{i_j}}$, and thus $D$ cannot be bounded.
\end{proof}

\begin{df}
  Given that they constitute a basis for the Von-Neumann topology, we
  will refer any $\OO$-submodule of the form
  (\ref{eqn:basicboundedsetequicar}) (resp.
  (\ref{eqn:basicboundedsetmixedcar})) as a {\it basic bounded
  $\OO$-submodule} of $K\roundb{t}$ (resp. $K\curlyb{t}$). 
\end{df}

\begin{corollary}
  If $\|\cdot\|: K\curlyb{t} \to \mathbb{R}$ is a seminorm which
  is bounded on bounded sets, then there is a real number $C > 0$ such
  that $\|t^i\| < C$ for every $i \in \ZZ$.
\end{corollary}

\begin{proof}
  Suppose that $\| \cdot \|$ is a seminorm such that the sequence of real
  numbers $( \|t^i\| )_{i \in \ZZ}$ is not bounded. Consider the bounded set
  \begin{equation*}
    \OO\curlyb{t} = \sum_{i \in \ZZ} \OO t^i,
  \end{equation*}
  and the sequence $( t^i )_{i \in \ZZ} \subset
  \OO\curlyb{t}$. The
  seminorm $\| \cdot \|$ is not bounded on $\OO_F$.
\end{proof}

Contrary to the situation in the equal characteristic case, in the mixed
characteristic setting we get the following.

\begin{corollary}
  \label{cor:curlybisnotbornological}
  The space $K\curlyb{t}$ is not bornological.
\end{corollary}

\begin{proof}
  It is enough to supply a seminorm which is bounded on bounded sets but not continuous.

  Consider the norm on $K\curlyb{t}$ given by
  \begin{equation}
    \label{eqn:normattachedtovalmixedchar}
    \left\| \sum_{i \in \ZZ} x_i t^i \right\| = \sup_{i \in \ZZ} |x_i|,
  \end{equation}
  which is the absolute value on $K\curlyb{t}$ related to the valuation
  $v_F$. If $B$ is a basic bounded set as in
  (\ref{eqn:basicboundedsetmixedcar}), then
  \begin{equation*}
    \sup_{x \in B} \| x \| = \sup_{i \in \ZZ} q^{-k_i},
  \end{equation*}
  and hence $\| \cdot \|$ is bounded on bounded sets. However, the norm
  $\| \cdot \|$ is not continuous on $K\curlyb{t}$ because
  \begin{equation*}
    \OO\curlyb{t} = \left\{ x \in K\curlyb{t};\; \| x \| \leq 1 \right\}
  \end{equation*}
  is closed but not open in $K\curlyb{t}$.
\end{proof}

\begin{rmk}
  \label{rmk:mixedcharsuppressconditionii}
  When defining the higher topology on $K\curlyb{t}$, an admissible
  seminorm was attached to a sequence $(n_i)_{i \in \ZZ} \subset \ZZ \cup
  \left\{ -\infty \right\}$ subject to two conditions:
  \begin{enumerate}
    \item The $n_i$ are bounded above.
    \item We have $n_i \to -\infty$ as $i \to \infty$.
  \end{enumerate}
  However, in the proof of Proposition \ref{prop:locconvmixedcase} we
  did not require to make use of condition (ii).

  Indeed, if we remove condition (ii) and allow all sequences
  $(n_i)_{i \in \ZZ}$ satisfying only condition (i), we obtain a locally
  convex topology. Let us describe it: on one hand, the norm on
  $K\curlyb{t}$ given by
  \begin{equation*}
    \sum_{i \in \ZZ} x_i t^i \mapsto \sup_{i \in \ZZ} |x_i|,
  \end{equation*}
  becomes continuous, as it corresponds to taking $n_i = 0$ for all
  $i \in \ZZ$. Hence, the resulting locally convex topology is both finer
  than the higher topology and finer than the complete discrete valuation
  topology. It is an immediate exercise to see that under such a topology
  the ring of integers $\OO\curlyb{t}$ is a bounded open lattice and this
  is equivalent to the locally convex topology being defined by a single
  seminorm \cite[Proposition 4.11]{schneider-non-archimedean-functional-analysis}. We conclude that the resulting topology is the complete
  discrete valuation topology.

  It is immediate to check that the complete discrete valuation topology
  on $K\curlyb{t}$ defines a Banach $K$-algebra structure with very nice
  analytic properties. However, it is unclear whether this structure is
  of any arithmetic interest.
\end{rmk}

\begin{prop}
  \label{prop:muisbounded}
  Let $F = K\roundb{t}$ or $K\curlyb{t}$. The multiplication map
  $\mu: F \times F \to F$ is bounded with respect to the product
  bornology on the domain.
\end{prop}

\begin{proof}
  Let $B_1 = \sum_{i \in \ZZ} \pp^{m_i} t^i$ and $B_2 = \sum_{i \in \ZZ}
  \pp^{n_j} t^j$ be two bounded $\OO$-submodules of $F$. The product
  bornology on $F \times F$ is generated by sets of the form $B_1 \times
  B_2$. We have that $\mu(B_1, B_2) = \sum_{k \in \ZZ} V_k t^k$ with
  $V_k = \sum_{k = i+j} \pp^{m_i}\pp^{n_j} = \sum_{k = i+j}
  \pp^{m_i+n_j}$. We distinguish cases.

  If $F = K\roundb{t}$, $m_i = \infty$ and $n_j = \infty$ if $i$ and $j$
  are small enough. In this case, the sum defining $V_k$ is actually
  finite and there is $l_k \in \ZZ \cup \left\{ \infty \right\}$ such
  that $V_k \subset \pp^{l_k}$. Moreover, we actually have $V_k = \left\{
  0 \right\}$ if $k$ is small enough and therefore $\mu(B_1, B_2) \subset
  F$ is bounded.

  If $F = K\curlyb{t}$, then there are integers $c$ and $d$ such that
  $m_i\geq c$ for all $i \in \ZZ$ and $n_j \geq d$ for all $j \in \ZZ$.
  This implies that $V_k \subset \pp^{c+d}$ for every $k$ and that it is
  bounded.
\end{proof}

\section{Complete, c-compact and compactoid \break$\OO$-submodules}
\label{sec:completeccompactcompactoidsubmodules}

In this section we will study relevant $\OO$-submodules of
$K\roundb{t}$ and $K\curlyb{t}$, including rings of integers and rank-2
rings of integers.

We start dealing with completeness of rings of integers.

\begin{prop}
  \label{prop:completeness}
  The rings of integers $K\squareb{t}$ and $\OO\curlyb{t}$ are complete
  \break$\OO$-submodules of $K\roundb{t}$ and $K\curlyb{t}$, respectively.
\end{prop}

In the case of $K\squareb{t} \subset K\roundb{t}$, the result follows
because $K\roundb{t}$ is complete and $K\squareb{t}$ is a closed subset.
However, it is also immediate to give an argument by hand.

\begin{proof}
  Let $I$ be a directed set and $\left( x_i \right)_{i \in I}$ a Cauchy
  net in the ring of integers. We distinguish cases below.
  
  In the case of $K\squareb{t}$, we write $x_i = \sum_{k\geq 0} x_{k,i} t^k$ 
  with $x_{k,i} \in K$. Since $(x_i)_{i \in I}$ is a Cauchy net in
  $\OO_F$, we have that $(x_{k,i})_{i \in I}$ is a Cauchy
  net in $K$ and hence converges to an element $x_k \in K$ for every
  $k \geq 0$. The element
  $x = \sum_{k \geq 0} x_k t^k$ is the limit of the Cauchy net.

  In the case of $\OO\curlyb{t}$, the procedure is very similar. We write
  $x_i = \sum_{k \in \ZZ} x_{i,k} t^k$ with $x_{i,k} \in \OO$. Since
  $\OO$ is complete and $(x_{i,k})_{i \in I}$ is a Cauchy net, it
  converges to an element $x_k \in \OO$ for every $k \in \ZZ$. It is elementary to check that
  as $k \to -\infty$, we have $x_k \to 0$ and therefore $x =
  \sum_{k \in \ZZ} x_k t^k$ is a well-defined element in $\OO\curlyb{t}$
  which is the limit of the Cauchy net.
\end{proof}

\begin{corollary}
  The rank-2 rings of integers of $K\roundb{t}$ and $K\curlyb{t}$ are complete.
\end{corollary}

\begin{proof}
  It follows from the previous proposition due to the fact that they are
  closed subsets of complete $\OO$-submodules.
\end{proof}

Next we will study rings of integers from the point of view of
c-compactness and compactoidness.

\begin{prop}
  \label{prop:Ksquarebisccpct}
  $K\squareb{t}$ is c-compact.
\end{prop}

\begin{proof}
  As a locally convex $K$-vector space, $K\squareb{t}$ is isomorphic to
  $K^\mathbb{N}$ (Proposition \ref{prop:toponKsquarebisprodtop}). The
  field $K$ is c-compact (Example \ref{exa:Kisccpt}). Finally,
  a product of c-compact spaces is c-compact (Proposition
  \ref{prop:prodccptisccpt}).
\end{proof}

\begin{corollary}
  \label{cor:ranktwoequicarisccpct}
  The rank-2 ring of integers of $K\roundb{t}$, $\OO + t K\squareb{t}$, is
  c-compact.
\end{corollary}

\begin{proof}
  After the previous proposition, the result follows from the fact that
  $\OO + t K\squareb{t} \subset K\squareb{t}$ is closed, as c-compactness 
  is hereditary for closed subsets \cite[Lemma 12.1.iii]{schneider-non-archimedean-functional-analysis}.
\end{proof}

\begin{corollary}
  \label{cor:ringsofintegersequicarnotcompactoid}
  The rings $K\squareb{t}$ and $\OO + t K\squareb{t}$ are not compactoid.
\end{corollary}

\begin{proof}
  This follows from the fact that they are both c-compact, unbounded,
  complete and Proposition \ref{prop:ccpctbddcptoidcmplt}.
\end{proof}

The compactoid submodules of a locally convex vector space
define a bornology. Since every compactoid submodule is bounded, in our
case it is important to decide which basic bounded submodules of
$K\roundb{t}$ and $K\curlyb{t}$ are compactoid.

Since $K\roundb{t}$ is a nuclear space, the class of bounded
$\OO$-submodules and compactoid $\OO$-submodules coincide
\cite[Proposition 19.2]{schneider-non-archimedean-functional-analysis}.

It is in any case easy to see that any basic bounded subset
\begin{equation*}
  B = \sum_{i \geq i_0} \pp^{k_i} t^i,\quad k_i \in \ZZ \cup \left\{
  \infty \right\}
\end{equation*}
is compactoid: suppose that $\Lambda = \sum_{i \in \ZZ} \pp^{n_i} t^i$ with $n_i \in
  \ZZ \cup \left\{ -\infty \right\}$ and such that for every $i > i_1$ we
  have $n_i = -\infty$. If $i_1 < i_0$ then $B \subset \Lambda$ and there is nothing to show.
  Otherwise, let $l_i = \min(n_i,k_i)$ for $i_0 \leq i \leq i_1$. Then
  \begin{equation*}
    B \subseteq \Lambda + \sum_{i = i_0}^{i_1} \OO \cdot \pi^{l_i} t^i,
  \end{equation*}
  which shows that it is compactoid.

\begin{corollary}
  The basic bounded $\OO$-submodules of $K\roundb{t}$ are c-compact.
\end{corollary}

\begin{proof}
  In view of Proposition \ref{prop:ccpctbddcptoidcmplt}, it is enough to
  show that a submodule $B$ as in the proof of the previous proposition
  is complete for nets. But the argument for showing completeness of such $\OO$-submodules
  is the same as in the proof of Proposition \ref{prop:completeness} and
  we shall omit it.
\end{proof}

In the case of $K\curlyb{t}$ there is a difference between bounded and
compactoid $\OO$-submodules. For the proof of the following proposition
we will consider the projection maps
\begin{equation*}
  \pi_j : K\curlyb{t} \to K, \quad \sum_{i \in \ZZ} x_i t^i \mapsto
  x_j,\quad j \in \ZZ.
\end{equation*}
These are examples of continuous nonzero linear forms on
$K\curlyb{t}$.

\begin{prop}
  \label{prop:basiccompactoidsubmodulesmixedchar}
  The only compactoid submodules amongst the basic bounded submodules of
  $K\curlyb{t}$ are the ones of the form
  \begin{equation}
    \label{eqn:compactoidsubmodulesmixedchar}
    B = \sum_{i \in \ZZ} \pp^{k_i} t^i,
  \end{equation}
  with $k_i \in \ZZ$ bounded below and such that $k_i \to \infty$ as
  $i \to -\infty$.
\end{prop}

\begin{proof}
  Let $B$ be a basic bounded submodule as in
  (\ref{eqn:compactoidsubmodulesmixedchar}), with the $k_i \in \ZZ \cup
  \left\{ \infty \right\}$ bounded below.

  On one hand, assume that $k_i \to \infty$ as $i \to -\infty$. Let $\Lambda =
  \sum_{i \in \ZZ} \pp^{n_i} t^i$ be an open lattice and assume that $B$
  is not contained in $\Lambda$, as otherwise there is nothing to prove.
  When $i \to \infty$, the $k_i$ are bounded below and $n_i \to -\infty$. Similarly,
  when $i \to -\infty$, $k_i \to \infty$ and the $n_i$ are bounded above.
  Hence, the following two statements are true:
  \begin{enumerate}
    \item There is an index $i_0$ such that for every $i < i_0$, $k_i
      \geq n_i$.
    \item There is an index $i_1$ such that for every $i > i_1$,
      $k_i \geq n_i$.
  \end{enumerate}
  We have $i_0 \leq i_1$, as otherwise $B$ is contained in $\Lambda$. Let
  $l_i = \min(k_i, n_i)$ for $i_0 \leq i \leq i_1$. Then we have
  \begin{equation*}
    B \subseteq \Lambda + \sum_{i = i_0}^{i_1} \OO \cdot \pi^{l_i} t^i,
  \end{equation*}
  which shows that $B$ is compactoid.

  On the other hand, suppose that the $k_i$ do not tend to infinity as $i \to
  -\infty$. In such case, there is a decreasing sequence
  $(i_j)_{j \geq 0} \subset \ZZ_{<0}$ such that $(k_{i_j})_{j \geq 0}$ is bounded above.
  Let $M \in \ZZ$ be such that $k_{i_j} < M$ for every $j \geq 0$.

  Let 
  \begin{equation*}
    \Lambda = \sum_{i < 0} \pp^M t^i + \sum_{i \geq 0} K t^i \subset K\curlyb{t},
  \end{equation*}
  which is an open lattice. Suppose that $x_1, \ldots, x_m \in
  K\curlyb{t}$ satisfy that $B \subseteq \Lambda + \OO x_1 + \cdots + \OO
  x_m$. We denote $x_l = \sum_{i \in \ZZ} x_{l,i} t^i$, with
  $x_{l,i} \in K$, for $1 \leq l \leq m$.

  We know that for $1 \leq l \leq m$, we have $x_{l,i} \to 0$ as
  $i \to -\infty$. Therefore, there is an index $j_0 \geq 0$ such that
  for every $j \geq j_0$ we have $v_K(x_{l,i_{j}}) > M$. Then we have
  \begin{equation*}
    \pi_{i_{j_0}}(B) \subseteq \pi_{i_{j_0}}(\Lambda + \OO x_1 + \cdots
    + \OO x_m),
  \end{equation*}
  from where we deduce
  \begin{equation*}
    \pp^{k_{i_{j_0}}} \subseteq \pp^M + \pp^{v_K(x_{1,i_{j_0}})} + \cdots
    + \pp^{v_K(x_{m,i_{j_0}})} = \pp^M.
  \end{equation*}
  However, this inclusion contradicts the fact that $k_{i_{j_0}} < M$.
\end{proof}

\begin{df}
  We will refer in the sequel to the $\OO$-submodules of the form
  (\ref{eqn:basicboundedsetequicar}) (resp.
  (\ref{eqn:compactoidsubmodulesmixedchar})) as {\it basic compactoid
  submodules} of $K\roundb{t}$ (resp. $K\curlyb{t}$).
\end{df}

We deduce several consequences of this result.

\begin{corollary}
  \label{cor:curlybisnotnuclear}
  The field $K\curlyb{t}$ is not a nuclear space.
\end{corollary}

\begin{proof}
  After the previous Proposition and (i) in Proposition
  \ref{prop:propsofnuclearspaces}, the result follows by observing that
  $K\curlyb{t}$ contains $\OO$-submodules, such as $\OO\curlyb{t}$, which
  are bounded but not compactoid.
\end{proof}

\begin{corollary}
  The basic compactoid submodules of $K\curlyb{t}$ are c-compact.
\end{corollary}

\begin{proof}
  Again, in view of Proposition \ref{prop:ccpctbddcptoidcmplt}, it is
  enough to show that the $\OO$-submodule $B$ as in
  (\ref{eqn:compactoidsubmodulesmixedchar}) is complete. The argument is the same as in the proof of Proposition
  \ref{prop:completeness} and we omit it.
\end{proof}

The proof of the following corollary is immediate after Proposition
\ref{prop:basiccompactoidsubmodulesmixedchar}.

\begin{corollary}
  \label{cor:ringsofintegersmixedcarnotcompactoid}
  $\OO\curlyb{t}$ and the rank-two ring of integers of
  $K\curlyb{t}$ are not compactoid nor c-compact.
\end{corollary}

\begin{proof}
  The fact that these rings are not compactoid follows from Proposition
\ref{prop:basiccompactoidsubmodulesmixedchar}. The fact that they are not
c-compact follows from the fact that, on top of not being compactoid,
they both are bounded and complete.
\end{proof}

\begin{corollary}
  \label{cor:muisboundedcompactoid}
  Multiplication $\mu: K\curlyb{t} \times K\curlyb{t} \to
  K\curlyb{t}$ is also bounded when $K\curlyb{t}$ is endowed with the
  bornology generated by compactoid $\OO$-submodules in the codomain, and
  the product of two copies of such bornology in the domain.
\end{corollary}

\begin{proof}
  Let $B_1 = \sum_{i \in \ZZ} \pp^{m_i} t^i$ and $B_2 = \sum_{j \in \ZZ}
  \pp^{n_j} t^j$ be two basic compactoid $\OO$-submodules of
  $K\curlyb{t}$; let $V_k$ as in the proof of Proposition
  \ref{prop:muisbounded}. Just like in the aforementioned proof, $V_k$ is
  a contained in a fractional ideal of $K$ and is therefore bounded.
  Moreover, it is possible to choose $l_k \in \ZZ \cup
  \left\{ \infty \right\}$ such that $l_k \to \infty$ as $k \to -\infty$
  and $V_k \subset \pp^{l_k}$; this proves
  that $\mu(B_1,B_2)$ is contained in a compactoid $\OO$-submodule of
  $K\curlyb{t}$.
\end{proof}

\section{Duality}
\label{sec:duality}

Let us describe some duality issues of
two-dimensional local fields when regarded as locally convex vector
spaces over a local field.

Much is known about the self-duality of the additive group
of a two-dimensional local field. From \cite[\textsection
3]{fesenko-aoas1}, if $F$ is a two-dimensional local field, once a
nontrivial continuous character
\begin{equation*}
  \psi: F \to S^1 \subset \mathbb{C}^\times
\end{equation*}
has been fixed, the group of continuous characters of the additive group
of $F$ consists entirely of characters of the form $\alpha \to
\psi(a\alpha)$, where $a$ runs through all elements of $F$. This result
is entirely analogous to the one-dimensional theory \cite[Lemma
2.2.1]{tate-thesis}.

In the case of
$K\roundb{t}$ and $K\curlyb{t}$, self-duality of the additive group
follows in an explicit way from
two self-dualities: that of the two-dimensional local field as a locally convex $K$-vector space, and
that of the additive group of $K$ as a locally compact abelian group.
Since the second is sufficiently documented \cite[\textsection 2.2]{tate-thesis}, let us focus on the first
one.

%It is in general a difficult problem for a locally convex vector space
%over a nonarchimedean field to exhibit nonzero continuous linear forms on
%the space. There is an affirmative answer in the case in which the base
%field is spherically complete, relying on the Hahn-Banach theorem
%\cite[Prop. 9.2]{schneider-non-archimedean-functional-analysis}. This is
%valid in our case, as $K$ is spherically complete.

We have already exhibited nontrivial continuous linear forms on a two-dimensional local
field. Let $F = K\roundb{t}$ or $K\curlyb{t}$; the map
\begin{equation}
  \pi_i : F \to K, \quad \sum x_j t^j \mapsto x_i
  \label{eqn:nonzerolinearformsona2dlf}
\end{equation}
is a continuous nonzero linear form for all $i \in \ZZ$.

Consider now the following map:
\begin{equation*}
  \gamma: F \to F',\quad x \mapsto \pi_x,
\end{equation*}
with
\begin{equation*}
  \pi_x: F \to K,\quad y \mapsto \pi_0(xy).
\end{equation*}
More explicitly, if $x = \sum x_i t^i$ and $y = \sum y_i t^i$, then
\begin{equation*}
  \pi_x(y) = \sum x_i y_{-i}.
\end{equation*}

The map $\gamma$ is well-defined because $\pi_x$, being the composition
of multiplication by a fixed element $F \to F$ and the projection
$\pi_0: F \to K$, is a continuous linear form. Besides that, $\gamma$ is
$K$-linear and injective.

%Let $\mathcal{B}$ denote the bornology
%of $F$ generated by compactoid $\OO$-submodules. If $F = K\roundb{t}$, let
%$\mathcal{B}$ be the collection of all bounded sets of $F$, so that
%$F'_{\mathcal{B}} = F'_b$. If $F = K\curlyb{t}$, let $\mathcal{B}$ be the
%family of bounded sets of $F$ generated by bounded sets of the form
%\begin{equation}
%  B = \sum_{i \in \ZZ} \pp^{k_i} t^i
%  \label{eqn:basicboundedsetmixeddual}
%\end{equation}
%with $k_i \in \ZZ$ bounded below and such that $k_i \to \infty$ when
%$i \to -\infty$. Note by comparison to Proposition 
%\ref{prop:basicboundedsetmixedcar} that the family of all bounded subsets
%of $F$ is strictly larger than $\mathcal{B}$.

\begin{rmk}
  Regarding topologies on dual spaces, we have that $K\roundb{t}'_c =
  K\roundb{t}'_b$ after Proposition \ref{prop:propsofnuclearspaces}.(i).
  However, the topology of
  $K\curlyb{t}'_c$ is strictly weaker than the
  one of $K\curlyb{t}'_b$: consider the seminorm
  \begin{equation*}
    |\cdot|_{\OO\curlyb{t}} : K\curlyb{t}' \to \mathbb{R}, \quad
    l \mapsto \sup_{x \in \OO\curlyb{t}} |l(x)|,
  \end{equation*}
  which is continuous with respect to the $b$-topology. If $|
  \cdot|_{\OO\curlyb{t}}$ was continuous with respect to the
  $c$-topology, there would be a basic compactoid submodule $B \subset
  K\curlyb{t}$ and a constant $C > 0$ such that
  $|l|_{\OO\curlyb{t}} \leq C |l|_B$ for all $l \in K\curlyb{t}'$.
  
  However, suppose that $B = \sum_{i \in \ZZ} \pp^{k_i} t^i$ with $k_i \to
  \infty$ as $i \to -\infty$. For any real
  number $C > 0$ there is an index $j \in \ZZ$ such that $Cq^{-k_j} < 1$. 
  This implies the inequality
  \begin{equation*}
    C |\pi_j|_B < |\pi_j|_{\OO\curlyb{t}}.
  \end{equation*}
  This shows that $|\cdot|_{\OO\curlyb{t}}$ is not
  continuous in the $c$-topology.
\end{rmk}

\begin{thm}
  \label{thm:selfduality}
  The map $\gamma: F \to F'_c$ is an isomorphism of locally convex $K$-vector spaces.
\end{thm}

Before we prove this result, we need an auxiliary result.

\begin{lem}
  \label{lem:dualbasis}
  Let $w \in F'$ and define, for every $i \in \ZZ$, $a_i = w(
  t^{-i})$. Then the formal sum $\sum a_i t^i$ defines an
  element of $F$.
\end{lem}

\begin{proof}
  We distinguish cases. If $F = K\roundb{t}$, it is necessary to show
  that $a_i = 0$ for all small enough indices $i$. In other words, that
  there is an index $i_0 \in \ZZ$ such that for every $i \geq i_0$ we
  have $w(t^i) = 0$. Without loss of generality, we may restrict
  ourselves to a continuous linear form $w : K\squareb{t}\to K$. In this
  case we get our result from the following isomorphisms: first
  $K\squareb{t} \cong K^\mathbb{N}$, second $(K^\mathbb{N})' \cong
  \bigoplus_\mathbb{N} K'$ 
  \cite[Theorem
  7.4.22]{perez-garcia-schikof-locally-convex-spaces-nonarchimedean-valued-fields}, 
  and third $K' \cong K$.

  In the case in which $F = K\curlyb{t}$, we need to show that the values
  $|a_i|$ for $i \in \ZZ$ are bounded and that $|a_i|\to 0$ as
  $i \to -\infty$. On one hand, the subset $\OO\curlyb{t} \subset F$ is
  bounded after Proposition \ref{prop:basicboundedsetmixedcar} and $t^i \in \OO\curlyb{t}$ for every $i \in \ZZ$. As
  $w$ is continuous, the set $w(\OO\curlyb{t}) \subset K$ is bounded and therefore
  the values $w(t^i)$ are bounded. On the other hand, the
  net $\left( t^i \right)_{i \in \ZZ}$ tends to zero in
  $K\curlyb{t}$ as $i \to \infty$. As $w$ is continuous, $a_i = w(t^{-i}) \to 0$ as
  $i \to -\infty$.
\end{proof}

\begin{proof}[Proof of Theorem \ref{thm:selfduality}]
  As explained above, the map $\gamma$ is
  well-defined, $K$-linear and injective.

%  \textit{Step 1.} Multiplication by a fixed element $F \to F$ is a
%  continuous map and projection to the 0-th coefficient is also
%  continuous. Hence, $\pi_x \in F'$ for every $x \in F$. Moreover,
%  $\gamma$ is linear and injective.

  Let $w \in F'$. Define $x = \sum_i a_i t^i \in F$
  with $a_i$ as in Lemma \ref{lem:dualbasis}. Then, for $y = \sum y_i t^i
  \in F$, we have
  \begin{equation*}
    w \left( \sum y_i t^i \right) = \sum y_i w(t^i) = \sum y_i a_{-i} = \pi_0(xy)
  \end{equation*}
  (the first equality follows from Remark \ref{rmk:powerseriesconverge}). Therefore, $w = \pi_x$ and the map $\delta$ is surjective.

  In order to show bicontinuity, let us first work out what
  continuity means in this setting. For any $\varepsilon > 0$ and $B$ a
  set in the bornology generated by compactoid submodules,
  we must show that there are $\delta >0$ and an
  admissible seminorm $\|\cdot\| : F \to K$ such that $\|x\| \leq \delta$ implies
  $|\pi_x|_B \leq \varepsilon$.

  Without loss of generality, we may replace $\varepsilon$ and $\delta$
  by integer powers of $q$, and the generic bounded set
  $B$ by a basic compactoid submodule of $F$, which is of the form
  (\ref{eqn:basicboundedsetequicar}) in the equal characteristic case or
  (\ref{eqn:compactoidsubmodulesmixedchar}) in the mixed characteristic case. For
  convenience, let us write
  \begin{equation*}
    B = \sum_{i \in \ZZ} \pp^{k_i} t^i, \quad k_i \in \mathbb{Z} \cup \left\{
    \infty \right\}
  \end{equation*}
  by allowing, in the equal characteristic case, $k_i = \infty$ for every
  small enough $i$.

  Now, let $n \in \ZZ$. We take $n_i = -k_{-i}$ for every $i \in \ZZ$.
  Because of the conditions defining $B$, the sequence $( n_i )_{i\in\ZZ}$
  defines an admissible seminorm $\| \cdot \|$ in both cases. Now, for $x = \sum
  x_i t^i$, we have that $\|x\| \leq q^n$ if and only if for every index
  $i\in \ZZ$ we have 
  \begin{equation}
    n_i - n \leq v_K(x_i). 
    \label{eqn:thmselfduality}
  \end{equation}
  Similarly, $|\pi_x|_B \leq q^n$ if and only if for every index $i \in 
  \ZZ$ we have
  \begin{equation}
    -k_{-i} - n \leq v_K(x_i).
    \label{eqn:thmselfduality2}
  \end{equation}
  By direct comparison and substitution between
  (\ref{eqn:thmselfduality}) and (\ref{eqn:thmselfduality2}), we have
  that with our choice of admissible seminorm $\| \cdot \|$,
  \begin{equation*}
    \|x\| \leq q^n \text{ if and only if } |\pi_x|_B \leq q^n,
  \end{equation*}
  which shows bicontinuity.
\end{proof}

\begin{rmk}
  In the mixed characteristic case we may ask ourselves if it is possible
  to exhibit any self-duality result involving $F'_b$, that is,
  topologizing the dual space according to uniform convergence over all
  bounded sets.

  It can be seen from the proof of Theorem \ref{thm:selfduality} that
  this is not the case. Any bornology $\mathcal{B}$ stronger than the
  one generated by compactoid submodules will stop the map $\gamma: F \to
  F'_{\mathcal{B}}$ from being continuous.

  We remark that if there were no other bounded sets in
  $K\curlyb{t}$ besides the ones generated by compactoid submodules, it would be possible 
  to show that such a locally convex vector space is bornological.

  From the failure of $K\curlyb{t}$ at being bornological we may deduce
  that Theorem \ref{thm:selfduality} is the best possible result.
\end{rmk}

By applying Theorem \ref{thm:selfduality} twice on $K\roundb{t}$, we
recover the fact that this locally convex space is reflexive. Indeed, we
have made this fact explicit via the choice of duality pairing:
\begin{equation*}
  K\roundb{t} \times K\roundb{t} \to K,\quad (x, y) \mapsto \pi_0(xy).
\end{equation*}

\begin{corollary}
  \label{cor:curlybisnotreflexive}
  The field $K\curlyb{t}$ is not reflexive.
\end{corollary}

\begin{proof}
  Since any reflexive space is barrelled (Proposition
  \ref{prop:reflexiveisbarrelled}), the
  result follows from Proposition \ref{prop:curlybisnotbarrelled}. 
\end{proof}

In order to conclude this section let us describe polars and pseudo-polars of the
$\OO$-submodules which we have studied in \textsection
\ref{sec:completeccompactcompactoidsubmodules}.

After Theorem \ref{thm:selfduality}, the topological isomorphism given by
$\gamma$ allows us to identify $F$ with $F'_c$, and in particular lets us
relate their $\OO$-submodules.

\begin{df}
  Let $F = K\roundb{t}$ or $K\curlyb{t}$. Let $A \subset F$ be an
  $\OO$-submodule. We let
  \begin{equation*}
    A^\gamma = \gamma^{-1}(A^p) \subset F
  \end{equation*}
  and refer to it, by abuse of language, as the {\it pseudo-polar} of $A$.
\end{df}

\begin{prop}
  \label{prop:Agamma}
  Consider the $\OO$-submodule 
  \begin{equation*}
    A = \sum_{i \in \ZZ} \pp^{k_i} t^i,\quad k_i \in \ZZ \cup \left\{
    \pm \infty \right\}
  \end{equation*}
  of $F=K\roundb{t}$ or $K\curlyb{t}$. Then, we have
  \begin{equation*}
    A^\gamma = \sum_{i \in \ZZ} \pp^{1-k_{-i}} t^i.
  \end{equation*}
\end{prop}

\begin{proof}
  Let $B = \sum_{i \in \ZZ} \pp^{1-k_{-i}} t^i$.

  On one hand, suppose $x = \sum x_i t^i \in B$. We have, for every
  $y = \sum y_i t^i \in A$,
  \begin{equation*}
    |\pi_x(y)| = \left| \sum x_{-i} y_{i}\right| \leq \sup |x_{-i}||y_i|
    = \sup q^{-1+k_i-k_i} < 1
  \end{equation*}
  and, therefore, $B \subseteq A^\gamma$.

  On the other hand, suppose that $x = \sum x_i t^i \in A^\gamma$. Then,
  by definition, we have
  \begin{equation*}
    |\pi_x(y)| < 1,\quad \text{ for any } y \in A.
  \end{equation*}
  In particular, let $y = \pi^{k_i} t^i$. Then the inequality
  \begin{equation*}
    |\pi_x(\pi^{k_i} t^i)| = |x_{-i} \pi^{k_i}| < 1
  \end{equation*}
  implies that $v_K(x_{-i}) \geq 1 - k_i$. Therefore $x_{-i} \in
  \pp^{1-k_i}$. Since our conclusion holds for any $i \in \ZZ$, we have
  that $x \in A^\gamma$ and, therefore, $B \subset A^\gamma$.
\end{proof}

After (iv) in Proposition \ref{prop:pseudopolar}, we may think of the
following corollary as a proof that the submodule $A$ in the statement of
the previous Proposition is closed, as it is equal to its pseudo-bipolar;
proofs to this result and the following two corollaries are immediate.

\begin{corollary}
  \label{cor:bipolar}
  For an $\OO$-submodule $A$ as in the previous Proposition, we have
  $A^{pp} = A$. \qed
\end{corollary}

\begin{corollary}
  We have $K\squareb{t}^\gamma = K\squareb{t}$. For the rank-2 ring of
  integers, we have $(\OO + t K\squareb{t})^\gamma = \pp +
  tK\squareb{t}$. \qed
\end{corollary}

%\begin{proof}
%  If we let
%  \begin{equation*}
%    k_i = \left\{ 
%    \begin{array}{ll}
%      -\infty,& i \geq 0,\\
%      \infty,& i < 0,
%    \end{array} \right.
%  \end{equation*}
%  then we have
%  \begin{equation*}
%    1-k_{-i} = 
%    \left\{ 
%    \begin{array}{ll}
%      \infty, & i < 0,\\
%      -\infty,& i \geq 0.
%    \end{array}
%    \right.
%  \end{equation*}
%  Similarly, the indices corresponding to $\OO + tK\squareb{t}$ are
%  \begin{equation*}
%    k_i = \left\{ 
%    \begin{array}{ll}
%      -\infty,& i > 0,\\
%      0,& i =0,\\
%      \infty,& i < 0,
%    \end{array} \right.
%  \end{equation*}
%  and for these we have
%  \begin{equation*}
%    1-k_{-i} = 
%    \left\{ 
%    \begin{array}{ll}
%      \infty, & i < 0,\\
%      1,& i = 0,\\ 
%      -\infty,& i > 0.
%    \end{array}
%    \right.
%  \end{equation*}
%\end{proof}

\begin{corollary}
  We have $(\OO\curlyb{t})^\gamma = \pp\curlyb{t}$. For the rank-2 ring
  of integers, we have $(\sum_{i < 0} \pp t^i + \sum_{i \geq 0} \OO
  t^i)^\gamma = \sum_{i \leq 0} \OO t^i + \sum_{i > 0} \pp t^i$. \qed
\end{corollary}

%\begin{proof}
%  As the previous Corollary, it is enough to write down the indices $k_i$
%  and to compute $1 - k_{-i}$ in each case.
%\end{proof}

By Proposition \ref{prop:pseudopolar} and Theorem \ref{thm:selfduality},
pseudo-polarity exchanges open lattices and basic compactoid submodules. Under
the characterization given by Proposition \ref{prop:Agamma}, the relation
is evident.

The same arguments exposed apply to compute that the polar of the
$\OO$-submodule $\sum_{i \in \ZZ} \pp^{k_i} t^i$, $k_i \in \ZZ \cup
\left\{ \pm \infty \right\}$ is $\sum_{i \in \ZZ} \pp^{-k_{-i}}t^i$. As
such, the polar of an open lattice is a compactoid lattice and
vice versa.

Let us write down a table with pseudo-polars and polars of relevant
$\OO$-submodules:
%\begin{table}

\begin{tabular}{l | l | l}
  $\mathbf{A}$ & $\mathbf{A^{\gamma}}$ & \textbf{polar of} $\mathbf{A}$ \\ \hline 
    $K\squareb{t}$ & $K\squareb{t}$ & $K\squareb{t}$ \\ \hline 
    $\OO + t K\squareb{t}$ & $\pp + tK\squareb{t}$ & $\OO + t
    K\squareb{t}$ \\ \hline
    $\OO\curlyb{t}$ & $\pp\curlyb{t}$ & $\OO\curlyb{t}$\\ \hline
    $\sum_{i < 0}\pp t^i + \sum_{i \geq 0} \OO t^i$ & $\sum_{i \leq 0}
    \OO t^i + \sum_{i > 0} \pp t^i$ & $\sum_{i \leq 0}\OO t^i + \sum_{i >
    0} \pp^{-1} t^i$ \\ \hline
    $\Lambda = \sum_{i \in \ZZ} \pp^{n_i} t^i$ &
    $B = \sum_{i \in \ZZ} \pp^{1-n_{-i}}t^i$ & $B = \sum_{i
    \in \ZZ} \pp^{-n_{-i}}t^i$ \\
    (open lattice) & (compactoid) & (compactoid) \\ \hline
    $B = \sum_{i \in \ZZ} \pp^{k_i} t^i$ &
    $\Lambda = \sum_{i \in \ZZ} \pp^{1-k_{-i}}t^i$ & $B = \sum_{i
    \in \ZZ} \pp^{-k_{-i}}t^i$ \\
    (basic compactoid) & (open lattice) & (open lattice) 
\end{tabular}
%  \label{tab:polars}
%\end{table}

\vskip .5cm

The isomorphism $\gamma: F \to F'_c$ is not unique, as it 
depends on choosing a nonzero linear form on $F$, which in our case
is $\pi_0$. For example, replacing $\pi_0$ by $\pi_1$ in the
definition of $\gamma$ would lead to an identical result. The actual
shape of $A^\gamma$, for a given
$\OO$-submodule $A \subset F$, depends heavily on $\gamma$. However, the
fact that polarity exchanges open lattices with compactoid submodules
does not depend on $\gamma$.

In conclusion, taking the pseudo-polar or polar is a self-map on the set of
$\OO$-submodules of $K\roundb{t}$ or $K\curlyb{t}$ which reverses
inclusions, gives basic compactoid submodules when applied to open lattices and
vice versa, and whose square equals the identity map when restricted to
closed $\OO$-submodules.

\section{General two-dimensional local fields}
\label{sec:generalcase}
%\todo{simplify discussion using Rmk after 1.3.2 in ihlf}

In the previous sections of this work we have developed a systematic
study of $K\roundb{t}$ and $K\curlyb{t}$ from the point of view of the
theory of locally convex spaces over $K$. Let us explain how the previous
results extend to a general characteristic zero two-dimensional local
field $K \into F$. The moral of the story is that we can link the higher
topology on $F$ to the constructions on $K\roundb{t}$ and $K\curlyb{t}$
that we have performed in the preceding sections of this work by
performing operations such as restriction of scalars along a finite
extension and taking finite products, and due to their finite nature, none
of these operations modifies the properties of the resulting locally
convex spaces.

Due to the difference in their structures,
we consider the equal characteristic and mixed characteristic cases
separately.

\subsection{Equal characteristic}
\label{subsec:generalcaseequi}

Assume that $K \into F$ is a two-dimensional local field and that
$\car F = \car \overline{F}$. In this case, as explained in \textsection
\ref{sec:categoryof2dlfs}, the choice of a field embedding
$\overline{F} \into F$ determines an isomorphism $F \cong
\overline{F}\roundb{t}$.

Denote the algebraic closure of $K$ in $F$ by $\Kt$. The extension
$\Kt | K$ is finite and $\Kt \into F$ is the only coefficient field of
$F$ which factors the field inclusion $K \into F$
\cite[Lemma 2.7]{morrow-explicit-approach-to-residues}, and this is the only
coefficient field of $F$ that we will take into account in our
constructions.

\begin{rmk}
  It is a well-known fact that in this case the higher topology of $F$
  depends on the choice of a coefficient field
  \cite[Example 2.1.22]{yekutieli-explicit-construction-grothendieck-residue-complex}. This is why we stress
  that in this work the only coefficient field we consider is $\Kt \into
  F$ because the field embedding $K \into F$ is given a priori.
\end{rmk}

The $\Kt$-vector space
$F\cong\Kt\roundb{t}$ is a complete, bornological, barrelled, reflexive
and nuclear locally convex space by direct application of Corollary
\ref{cor:propsofequicar}. The higher topology on $F$ only depends on the
choice of the embedding $\Kt \into F$ and, therefore, does not change by
restriction of scalars along $K \into \Kt$.

Let us explain this fact with more detail. On one hand, all open lattices
$\Lambda$ are $\OO_{\Kt}$-modules and
hence also $\OO$-modules by restriction of scalars. On the other hand, if
$x \in F$, there is a positive power of $\pi_{\Kt}$ that maps $x$ to
$\Lambda$ by multiplication. Uniformizers of $K$ and $\Kt$ may be chosen
to be related by the ramification degree: $\pi = \pi_{\Kt}^{e\left( \Kt | K
\right)}$. Therefore, there is a positive power of $\pi$ which maps $x$
to $\Lambda$ by multiplication and we deduce local convexity over $K$.

The absolute value on
$\Kt$ restricts to the absolute value of $K$ and therefore Corollary
\ref{cor:seminormsequi} describes the admissible seminorms of $K \into F$ without any changes.

Moreover, after Proposition \ref{prop:toponKsquarebisprodtop} we have
that the higher topology on $\Kt\roundb{t}$ agrees with the strict 
inductive limit topology given by
\begin{equation*}
  \Kt\roundb{t} = \bigcup_{i \in \ZZ} t^i \Kt \squareb{t},
\end{equation*}
which is also a union of $K$-vector spaces. Since the extension $\Kt
\vert K$ is finite, we also have that $\Kt\squareb{t}$ is isomorphic to a
product of countably many copies of $K$ and is therefore a Fr\'echet
$K$-vector space. Hence, we get that $F$ is an LF-space over $K$ and in
particular we may deduce from Proposition \ref{cor:propsofequicar} that
$F$ is a complete, bornological, barrelled, reflexive and nuclear
$K$-vector space.

Because admissible seminorms do not change after restricting scalars to $K$, 
Proposition \ref{prop:basicboundedsetequicar} describes a basis of bounded
$\OO$-submodules of $F$. These are complete, and from nuclearity we
deduce that the classes of bounded $\OO$-submodules and compactoid
$\OO$-submodules of $F$ agree.

Since $\Kt$ is a finite dimensional $K$-vector space, it is c-compact.
The ring of integers $\OO_F = \Kt\squareb{t}$ is therefore c-compact,
being isomorphic to a product of copies of $\Kt$. It is unbounded,
complete and not compactoid after Proposition
\ref{prop:ccpctbddcptoidcmplt}.
Similarly, the rank-2 ring of integers of $F$ shares all these properties
with $\OO_F$.

Regarding duality, the fact that the map $\gamma: F \to F'_c$ is an
isomorphism of locally convex spaces does not change when we restrict
scalars to $K$. Explicit nonzero linear forms 
$F \to K$ may be constructed by composing the maps $\pi_i: F \to \Kt$ as
in (\ref{eqn:nonzerolinearformsona2dlf}) with
$\Tr_{\Kt|K}$.

\paragraph{Problem.} It is relevant to decide whether the class of
bounded sets of $F$ changes along with the change of vector space
and locally convex structures associated to the choice of a different
coefficient field.

%\begin{rmk}
%  It is a relevant question to decide whether the bornology of $F$ might
%  change along with the higher topology if we choose a different coefficient
%  field. The answer is that the bornology also changes in such case
%  because $F$ is bornological.
%
%  If the bornology is given, the higher topology is the strongest locally
%  convex topology on $F$ which yields the given collection of bounded
%  sets. If the bornology did not depend on choosing a coefficient field,
%  neither would the topology. This is not the case. 
%  %More than a fault of
%  %the theory, this is an important difference depending on $\car
%  %\overline{F}$ which, in the words of I. S. Cohen, lies in the nature of things.
%\end{rmk}

\subsection{Mixed characteristic}

If $\car F \neq \car \overline{F}$, then, as explained in \textsection
\ref{sec:categoryof2dlfs}, there is a unique field embedding $\QQ_p \into
F$. If we denote the algebraic closure of $\QQ_p$ in $F$ by $\Kt$, the
field inclusion $K \into F$ may be factored into the following diagram of
field embeddings
\begin{equation*}
  \xymatrix{
    & K\curlyb{t} \ar[r] & \Kt\curlyb{t} \ar[r] & F \\
    \QQ_p \ar[r] &  K \ar[u] \ar[r] & \Kt  \ar[u]
  },
\end{equation*}
in which all horizontal arrows are finite extensions.

The inclusions $K\into K\curlyb{t}$ and $\Kt \into \Kt\curlyb{t}$
correspond to the situation we have been dealing with in the preceding
sections of this work. Let $n = \left[ F :\Kt\curlyb{t} \right]$.

As $\Kt$-vector spaces, we have
\begin{equation*}
  F \cong \Kt\curlyb{t}^n.
\end{equation*}
The higher topology on $F$ may be defined as the product topology on $n$
copies of the higher topology on
$\Kt\curlyb{t}$ \cite[1.3.2]{madunts-zhukov-topology-hlfs}. Furthermore, it
does not depend on any choices of subfields $\Kt\curlyb{t} \subset F$
\cite[\textsection 1]{kato-existence-theorem-hlfs}. Hence, since the
product topology on a product of locally convex vector spaces is 
locally convex, the inclusion
$\Kt \into F$ gives a locally convex $\Kt$-vector space. Let us first
study this space, and later restrict scalars along the finite extension
$K \into \Kt$.

We may describe
the family of open lattices or, equivalently, continuous seminorms, from
the corresponding lattices or seminorms for $\Kt\curlyb{t}$ and
Proposition
\ref{prop:producttopoflocconv}.

The situation for the ring of integers $\OO_F$ is as follows. If we denote
$\OOt = \OO_{\Kt}$, the inclusion
$\OOt\curlyb{t} \into \OO_F$ turns $\OO_F$ into a rank-$n$ free
$\OOt\curlyb{t}$-module. Therefore the subspace topology on $\OO_F \subset
F$ coincides with the product topology on $\OO_F \cong
\OOt\curlyb{t}^n$. From here, it is possible to show that $\OO_F$ is a
bounded and complete $\OOt$-submodule of $F$ which is neither c-compact
nor compactoid. It is however closed, but not open, and this proves that
$F$ is not barrelled. The norm attached to the valuation $v_F$ provides a
example of a seminorm which is bounded on bounded sets but not
continuous, as its unit ball, $\OO_F$, is not open. Hence $\Kt \into F$ is not bornological.

From the self-duality of $\Kt\curlyb{t}$, we obtain a chain of
isomorphisms of locally convex $\Kt$-vector spaces
\begin{equation*}
  F'_c \cong \left( \Kt\curlyb{t}^n \right)'_c \cong \left(
  \Kt\curlyb{t}'_c \right)^n \cong \Kt\curlyb{t}^n \cong F,
\end{equation*}
which shows that $F$ is also self-dual. Explicit nonzero linear forms may
be constructed in this case composing the trace map
$\Tr_{F|\Kt\curlyb{t}}$ with the maps $\pi_i: \Kt\curlyb{t} \to \Kt$ as in
(\ref{eqn:nonzerolinearformsona2dlf}). Finally, as $\OO_F$ is a bounded
$\OO$-submodule which is not compactoid, we deduce that $\Kt \into F$ is not
nuclear.

In order to conclude our discussion, we need to verify that the
properties of $K \into F$ agree with the ones of $\Kt \into F$. The
discussion is very similar to what has already been discussed in
\textsection \ref{subsec:generalcaseequi}. In this case, the higher
topology on $F$ is known not to depend on any choices
\cite[\textsection 1]{kato-existence-theorem-hlfs}.
Similarly to what happens in the equal characteristic case,
$\OOt$-lattices are also $\OO$-lattices after restriction of scalars and
therefore the sets of open lattices and admissible seminorms for
$\Kt \into F$ and $K \into F$ coincide. This also implies that the
collections of bounded sets agree.

Suppose now that $B \subset F$ is a compactoid $\OOt$-module. Given any
open lattice $\Lambda \subset F$, there are $x_1, \ldots, x_m \in F$ such
that $B \subset \Lambda + \OOt x_1 + \cdots + \OOt x_m$. We have that
$\OOt$ is a free $\OO$-module of rank $e = e(\Kt \vert K)$ and, in
particular, there are $y_1, \ldots, y_e \in \OOt$ such that $\OOt =
\OO y_1 + \cdots + \OO y_e$. Hence, we have
\begin{equation*}
  B \subseteq \Lambda + \sum_{i = 1}^m \left( \sum_{j = 1}^e \OO y_j x_i \right)
\end{equation*}
and therefore $B$ is a compactoid $\OO$-submodule of $F$.

Reciprocally, if $B$ is an $\OOt$-module which is a compactoid
$\OO$-submodule of $F$ after restriction of scalars, then for any
open lattice $\Lambda$ there are finitely many $x_1, \ldots, x_m \in F$
such that $B \subseteq \Lambda + \OO x_1 + \cdots + \OO x_m$. But then we
have inclusions
\begin{equation*}
  B \subseteq \Lambda + \OO x_1 + \cdots + \OO x_m \subseteq \Lambda + \OOt
  x_1 + \cdots + \OOt x_m
\end{equation*}
and, therefore, $B$ is also compactoid as an $\OOt$-module.

It is for these reasons that $\OO_F$ is an $\OO$-lattice which is closed
and complete but not open, bounded but not compactoid. Hence $K \into F$
is neither barrelled nor nuclear. Since any reflexive space is barrelled
(Proposition \ref{prop:reflexiveisbarrelled}), $F$ cannot be reflexive. The norm attached to the lattice
$\OO_F$ is bounded on bounded sets but not continuous since $\OO_F$ is
not open and, therefore, $K \into F$ is not bornological.

Finally, the isomorphisms of locally convex $\Kt$-vector spaces
\begin{equation*}
  F'_c \cong \left( \Kt\curlyb{t}^n \right)'_c \cong \left(
  \Kt\curlyb{t}'_c \right)^n \cong \Kt\curlyb{t}^n \cong F
\end{equation*}
turn into isomorphisms of locally convex $K$-vector spaces after
restriction of scalars. Finally, we construct explicit non-zero linear
forms on $F$ by, for example, taking
\begin{equation*}
  \Tr_{\Kt \vert K} \circ \pi_i \circ \Tr_{F\vert \Kt\curlyb{t}}: F \to \Kt\curlyb{t}
  \to \Kt \to K, \quad i \in \ZZ.
\end{equation*}

\section{A note on the archimedean case}
\label{sec:archcase}

Let $\KK = \mathbb{R}$ or $\mathbb{C}$. We will denote by $| \cdot |$
either the usual absolute value on $\mathbb{R}$, or the module on
$\mathbb{C}$. %Let us fix the usual embedding $\mathbb{R} \subset \mathbb{C}$ (this is equivalent to choosing a square root of $-1$). In particular, $| \cdot |$ is preserved under this embedding.

In this section we will consider the study of archimedean two-dimensional local fields. An archimedean two-dimensional local field is a complete discrete valuation field $F$ whose residue field is an archimedean (one-dimensional) local field. Hence, we have a non-canonical isomorphism $F \cong \KK\roundb{t}$ for one of our two choices of $\KK$. Once an inclusion of fields $\KK \subset F$ has been fixed and $t$ has been chosen, a unique such isomorphism is determined.

The theory of locally convex vector spaces over $\KK$ was developed much
earlier than the analogous non-archimedean theory and is well explained
in, for example,
\cite{kothe-topological-vector-spaces-I}. Let $V$ be a $\KK$-vector space and $C \subseteq V$. The subset $C$ is said to be convex if for any $v,w \in C$, the segment
\begin{equation*}
  \left\{ \lambda v + \mu w; \lambda, \mu \in \mathbb{R}_{\geq 0}, \lambda + \mu = 1 \right\}
\end{equation*}
is contained in $C$. The subset $C$ is said to be absolutely convex if,
moreover, we have $\lambda C \subseteq C$ for every $\lambda \in
\mathbb{K}$ such that $|\lambda| \leq 1$.

We may associate a seminorm $p_C$ to any convex subset $C \subseteq V$ by the rule
\begin{equation*}
  p_C : V \to \mathbb{R}, \quad x \mapsto \inf_{\rho > 0, \; x \in \rho C} \rho.
\end{equation*}
This seminorm satisfies the usual triangle inequality, but not the ultrametric inequality.

\begin{df}
  \label{df:locconvarch}
  The $\KK$-vector space $V$ is said to be locally convex if it is a topological vector space such that its topology admits a basis of neighbourhoods of zero given by convex sets.
\end{df}

It may be shown that if $V$ is locally convex its filter of
neighbourhoods of zero also admits a basis formed by absolutely convex
subsets \cite[\textsection 18.1]{kothe-topological-vector-spaces-I}.

The higher topology on $\KK\roundb{t}$ is defined following the procedure outlined in \textsection \ref{sec:equicarcase}. In this case, we consider the disks of $\KK$ centered at zero and of rational radius; this defines a countable basis of convex neighbourhoods of zero for the euclidean topology on $\KK$. Denote
\begin{equation*}
  D_\rho = \left\{ a \in \KK;\; |a| < \rho \right\},\quad \rho \in \QQ_{>0} \cup \left\{ \infty \right\}.
\end{equation*}

Given a sequence $\left( \rho_i \right)_{i \in \ZZ} \subset \QQ_{>0} \cup \left\{ \infty \right\}$ such that there is an index $i_0$ satisfying that $\rho_i = \infty$ for all $i \geq i_0$, consider the set
\begin{equation}
  \label{eqn:basicopennhoodofzeroarchcase}
  \mathcal{U} = \sum_{i \in \ZZ} D_{\rho_i} t^i \subset F.
\end{equation}

The sets of the form (\ref{eqn:basicopennhoodofzeroarchcase}) form a basis of neighbourhoods of zero for the higher topology on $F$.

\begin{prop}
  The higher topology on $F$ is locally convex, in the sense of Definition \ref{df:locconvarch}.
\end{prop}

\begin{proof}
  As the discs $D_{\rho_i}$ are convex, given two elements $x,y \in \mathcal{U}$, it is easy to check that the segment
  \begin{equation*}
    \left\{ \lambda x + \mu y;\; \lambda, \mu \in \mathbb{R}_{\geq 0};\; \lambda + \mu = 1 \right\}
  \end{equation*}
  is contained in $\mathcal{U}$ by checking on each coefficient separately.

  Thus, the basis of open neighbourhoods of zero described by the sets of the form (\ref{eqn:basicopennhoodofzeroarchcase}) consists of convex sets, and hence the higher topology on $F$ is locally convex.
\end{proof}

As we have done in the rest of cases, we may now describe the higher topology in terms of seminorms.

\begin{prop}
  Let $k \in \ZZ$. Given a sequence $\rho_i \in \QQ_{>0} \cup \left\{ \infty \right\}$ for every $i \leq k$, such that $\rho_k < \infty$, consider the seminorm
\begin{equation}
  \label{eqn:highertopinseminormsarch}
  \| \cdot \|: \KK\roundb{t} \to \mathbb{R},\quad \sum_{i \geq i_0} x_i t^i \mapsto \max_{i \leq k} \frac{|x_i|}{\rho_i},
\end{equation}
having in mind the convention that $a/\infty = 0$ for every $a \in \mathbb{R}_{\geq 0}$.
The higher topology on $F$ is defined by the set of seminorms specified by (\ref{eqn:highertopinseminormsarch}).
\end{prop}

\begin{proof}
  We will show that the seminorm $\| \cdot \|$ defined by (\ref{eqn:highertopinseminormsarch}) is the gauge seminorm attached to the basic open neighbourhood of zero $\mathcal{U}$ given by (\ref{eqn:basicopennhoodofzeroarchcase}).

  Let $x = \sum_{i \geq i_0} x_i t^i \in F$ and $\rho > 0$. If $k < i_0$, we may take $\rho = 0$ and deduce that $q(x) = 0$.
  
  Otherwise, $x \in \rho\: \mathcal{U}$ if and only if $x_i \in \rho D_{\rho_i}$ for every $i_0 \leq i \leq k$.

  From this, we may deduce that $x \in \rho\: \mathcal{U}$ if and only if 
  \begin{equation}
    \label{eqn:inequalityinproofarchcase}
    \frac{|x_i|}{\rho_i} < \rho \quad \text{for every } i_0 \leq i \leq k.
  \end{equation}

  Finally, the infimum value of $\rho$ satisfying (\ref{eqn:inequalityinproofarchcase}) is precisely the maximum of the values $|x_i|/\rho_i$ for $i_0 \leq i \leq k$.
\end{proof}

We have described the higher topology on $\mathbb{K}\roundb{t}$ in a
fashion that matches what has been done in the previous sections.
However, this locally convex space often arises in functional analysis
in the following way. We write
\begin{equation*}
  \mathbb{K}\roundb{t} = \cup_{i \in \mathbb{N}} t^{-i}.
  \mathbb{K}\squareb{t},
\end{equation*}

Each component in the union is isomorphic to
$\mathbb{K}^\mathbb{N}$, topologized using the product topology, and
the limit acquires the strict inductive limit locally convex topology. 

It is known that $\mathbb{K}\squareb{t}$ is a Fr\'echet space, that is,
complete and metrizable. As such, the 
two-dimensional local field $\mathbb{K}\roundb{t}$ is an LF-space and
many of its properties may be deduced from the general theory of
LF-spaces, see for example \cite[\textsection 19]{kothe-topological-vector-spaces-I}. In particular, $\mathbb{K}\roundb{t}$ is complete,
bornological and nuclear.

\section{A note on the characteristic $p$ case}
\label{sec:charpcase}

Let $k = \mathbb{F}_q$ be a finite field of characteristic $p$. In this section we will consider the two-dimensional local field $F = k\roundb{u}\roundb{t}$. It is a vector space both over the finite field $k$ and over the local field $k\roundb{u}$.

The higher topology on $F$ may be dealt with in two ways from a linear
point of view. The first approach was started by Parshin
\cite{parshin-lCFT}, and it regards $F$ as a
$k$-vector space. In this approach, $k$ is regarded as a discrete
topological field and the tools used are those of linear topology, see
\cite[\textsection 1]{kapranov-semiinfinite-symmetric-powers} for an
account. Linear topology was first introduced by Lefschetz
\cite{lefschetz-algebraic-topology}.

The work developed in the previous sections of this work may be applied
and we may regard $F$ a locally convex $k\roundb{u}$-vector space. In this section we will explain that in this case we have obtained nothing new.

A topology on a $k$-vector space is said to be linear if the filter of
neighbourhoods of zero admits a collection of linear subspaces as a
basis. A linearly topological vector space $V$ is said to be linearly
compact if any family $A_i \subset V$, $i \in I$ of closed affine
subspaces such that $\bigcap_{i \in J} A_i \neq \emptyset$ for any finite
set $J \subset I$, then $\bigcap_{i \in I} A_i \neq \emptyset$. Finally,
a linearly topological vector space is locally linearly compact if it has a basis of neighbourhoods of zero formed by linearly compact open subspaces.

Let $\Vect_k$ be the category of linearly topological $k$-vector spaces.
Similarly, let $\Vect_{k\roundb{u}}$ be the category of locally convex
$k\roundb{u}$-vector spaces.

\begin{prop}
  The rule
  \begin{equation}
    \Vect_{k\roundb{u}} \to \Vect_k,
    \label{eqn:loc-conv-lin-top}
  \end{equation}
  which restricts scalars on $k\roundb{u}$-vector spaces along the
  inclusion $k \hookrightarrow k\roundb{u}$ and preserves topologies and
  linear maps, is a functor.
\end{prop}

\begin{proof}
  Let $V$ be a locally convex $k\roundb{u}$-vector space, and let
  $\Lambda$ denote an open lattice. As the lattice $\Lambda$ is an
  $\OO_{k\roundb{u}}$-module and we have the inclusion $k \hookrightarrow
  \OO_{k\roundb{u}} = k\squareb{u}$, it is also a $k$-vector space by restriction of scalars.

  As the collection of open lattices $\Lambda$ is a basis for the filter of neighbourhoods of zero, $V$ is a linearly topological $k$-vector space and the first part of the proposition follows.
\end{proof}
 
There is a strong analogy between the concepts of c-compactness for
locally convex $k\roundb{u}$-vector spaces and linear compactness for
linearly topological $k$-vector spaces; both definitions agree if in
Proposition \ref{prop:ccompactness} we translate the words {\it closed
convex subspace} by {\it closed affine subspace}.

However, it is not true in general that restriction of scalars on a
c-compact $k\roundb{u}$-vector space yields a linearly compact $k$-vector
space: $k\roundb{u}$, being spherically complete, is a c-compact
$k\roundb{u}$-vector space
\cite[\textsection 12]{schneider-non-archimedean-functional-analysis} which is not a linearly compact $k$-vector
space.

The lack of an embedding of a finite field into a characteristic zero
two-dimensional local field makes the linear topological approach
unavailable in that setting; the locally convex approach to these fields 
is therefore to be regarded as analogous to the linear approach in
positive characteristic. Similarly, the language of locally convex spaces
is to be regarded as one which unifies the approach to the
zero characteristic and positive characteristic cases.

\section{Future work}
\label{sec:future}

We outline some directions which we consider interesting to explore 
in order to apply and extend the results in this work.

\paragraph{$\OO$-linear locally convex approach to higher topology.} In
this work we have been able to deduce many properties about
$K\roundb{t}$, either in an explicit or implicit way, from the fact that
it is an LF-space, i.e.: an inductive limit of Fr\'echet spaces. This
is not the case in mixed characteristics: the field $K\curlyb{t}$ is not
a direct limit of nice $K$-vector spaces. It is, however, a direct limit
of $\OO$-modules by construction.

The development of a theory of locally convex $\OO$-modules, with
topologies defined by seminorms, and the constructions arising within
that theory, particularly those of initial and final locally convex 
topologies, would allow us to recover on one hand the results we have
established for $K\roundb{t}$, and on the other hand they would let us
describe $K\curlyb{t}$ as a direct limit of perhaps {\it nice}
$\OO$-modules; this could be an extremely helpful contribution to the
study of mixed characteristic two-dimensional local fields.

\paragraph{Generalization to higher local fields.} If $F$ is a
characteristic zero $n$-dimensional local field, then it is possible to
exhibit a field embedding $K \into F$ and treat $F$ as a $K$-vector
space. A higher topology on $F$ may be constructed inductively using the same
procedures outlined at the beginning of \textsection
\ref{sec:highertopislocconv}, see \cite{madunts-zhukov-topology-hlfs}. Therefore, it may be shown that these
topologies define locally convex structures over $K$. Although the
situation is slightly more complex, a systematic study of the functional
theoretic properties of these locally convex spaces would be interesting
to develop. The first steps in this direction have been taken in
\cite{camara-lcshlfs}.

\paragraph{Study of $\LL(F)$.} As we have explained, the ring of continuous 
$K$-linear endomorphisms of a two-dimensional local field can be topologized and 
studied from a functional analytic point of view. It contains several
relevant two-sided ideals defined by imposing certain finiteness
conditions to endomorphisms. The most important of such ideals is the
subspace of nuclear maps. Nuclear endomorphisms of a locally convex space play a
distinguished role in the study of the properties of such space. In
particular, the usual trace map on finite-rank operators extends by
topological arguments to the subspace of nuclear endomorphisms.
Establishing a characterization of nuclear endomorphisms of
two-dimensional local fields is an affordable goal.

\paragraph{Multiplicative theory of two-dimensional local fields.}
Multiplication $\mu: F \times F \to F$ on a two-dimensional field $F$ is not
continuous as explained in \textsection \ref{sec:firsttopoprop}. It is a
well-known fact that the map $\mu$ is sequentially continuous, and the
sequential topological properties of higher topologies have been studied
and applied successfully to higher class field theory
\cite{fesenko-sequential-topologies} and to topologize sets of
rational points of schemes over higher local fields
\cite{camara-top-rat-pts-hlfs-arxiv}. 

However, for any $x \in F$, the linear maps
\begin{align*}
  \mu(x,\cdot)&: F \to F,\\
  \mu(\cdot, x)&: F \to F
\end{align*}
are continuous. This means that, in the terms of
\cite[\textsection 17]{schneider-non-archimedean-functional-analysis},
$\mu$ is a separately continuous bilinear map and therefore induces a
continuous linear map of locally convex spaces
\begin{equation}
  \label{eqn:multiplicationoninductivetensorproduct}
  \mu: F \otimes_{K, \iota} F \to F,
\end{equation}
where $F \otimes_{K, \iota} F$ stands for the tensor product $F \otimes_K
F$ topologized using the inductive tensor product topology.

This suggests that besides the applications of the theory of
semitopological rings to the study of arithmetic properties of higher local fields \cite{yekutieli-explicit-construction-grothendieck-residue-complex}, we have the
following new approach to the topic: a two-dimensional local field $F$ 
is a locally convex $K$-vector space endowed with a continuous linear 
map $\mu: F \otimes_{K,\iota} F \to F$ satisfying the usual axioms of
multiplication. 

After Proposition \ref{prop:muisbounded} and Corollary
\ref{cor:muisboundedcompactoid}, another possible way to look at
a two-dimensional local field and deal with its multiplicative structure
is as a bornological $K$-algebra, that is: $F$ is a $K$-algebra endowed
with a bornology (that generated by bounded submodules or compactoid
submodules) such that all $K$-algebra operations
\begin{align*}
  \sigma &: F \times F \to F \quad \text{(addition)}, \\
  \varepsilon &: K \times F \to F \quad \text{(scalar multiplication)}, \\
  \mu &: F \times F \to F
\end{align*}
are bounded.

It is interesting to decide if the arithmetic properties of $F$ can be
recovered from these contexts, and it would even more interesting to establish new
connections between this functional analytic approach to higher topology
and the arithmetic of $F$.

\paragraph{Functional analysis on adelic rings and modules over them.}
There are several two-dimensional adelic objects which admit a formulation as a
restricted product of two-dimensional local fields and their rings of
integers, which in our characteristic zero context were
introduced by Beilinson \cite{beilinson-residues-adeles} and Fesenko
\cite{fesenko-aoas2} (see \cite[\textsection 8]{morrow-intro-hlf} for a
discussion of the topic). From what we have exhibited in
this work, at least in dimension two, these adelic objects may be studied
using the theory of locally convex spaces, archimedean or nonarchimedean.

\paragraph{Topological approach to higher measure and integration.} The
study of measure theory, integration and harmonic analysis on
two-dimensional local fields is an interesting problem. A theory of 
measure and integration has been developed on two-dimensional
local fields $F$ by lifting the Haar measure of the local field
$\overline{F}$ \cite{fesenko-aoas1},
\cite{morrow-integration-on-valuation-fields}. This theory relies heavily
on the relation between $F$ and $\overline{F}$. The approach to measure and
integration on $F$ using the functional theoretic tools arising from the
relation between $F$ and $K$ could yield an alternative integration
theory.

\def\cprime{$'$} \def\cprime{$'$}

\vskip 1cm
\begin{flushright}
	{\bf Alberto C\'amara}\\
	\footnotesize{
	{\it School of Mathematical Sciences\\
	University of Nottingham\\
	University Park\\
	Nottingham\\
	NG7 2RD\\
	United Kingdom\\}}
	\url{http://www.maths.nottingham.ac.uk/personal/pmxac}\\
	\href{mailto:pmxac@nottingham.ac.uk}{{\tt pmxac@nottingham.ac.uk}}
\end{flushright}
% Important: Do not put any empty line here.

\end{document}